\documentclass[12pt]{amsart}
\usepackage[utf8]{inputenc}

\usepackage{amsmath}%
\usepackage{amsfonts}%
\usepackage{amssymb}%
\usepackage{amsthm}
\usepackage[margin=1in]{geometry}
\usepackage{enumerate}

\usepackage{thmtools}
\usepackage{thm-restate}

\usepackage{hyperref}
\usepackage{cleveref}

\allowdisplaybreaks

\declaretheorem[name=Theorem,within=section]{theorem}
\declaretheorem[name=Corollary,sibling=theorem]{corollary}
\declaretheorem[name=Lemma,sibling=theorem]{lemma}
\declaretheorem[name=Proposition,sibling=theorem]{proposition}
\declaretheorem[name=Claim,within=section]{claim}

\declaretheorem[name=Question, sibling=theorem, style=definition]{question}
\declaretheorem[name=Definition,sibling=theorem, style=definition]{definition}

\newcommand{\isom}{\cong}
\mathchardef\mhyphen="2D

\newcommand{\mcV}{V}
\newcommand{\mcU}{U}
\newcommand{\mcW}{W}

\newcommand{\mcP}{\mathcal{P}}

\newcommand{\Xspace}{X}
\newcommand{\Gspace}{G}
\newcommand{\Yspace}{Y}

\newcommand{\mcC}{\mathcal{C}}

\newcommand{\BPi}{\mathbf{\Pi}}
\newcommand{\BSigma}{\mathbf{\Sigma}}
\newcommand{\BDelta}{\mathbf{\Delta}}

\newcommand{\oa}{a}
\newcommand{\ob}{b}
\newcommand{\oc}{c}
\newcommand{\od}{d}
\newcommand{\ooe}{e}
\newcommand{\of}{f}

\newcommand{\bfx}{\mathbf{x}}

\newcommand{\smf}{\smallfrown}
\newcommand{\Proj}{\text{Proj}}

\begin{document}
\title[Non-Archimedean TSI Polish groups]{Non-Archimedean TSI Polish groups and their potential Borel complexity spectrum }
\author{Shaun Allison}
\address{Department of Mathematical Sciences, Carnegie Mellon University, Pittsburgh, PA 15213}
\email{shaunpallison@gmail.com}
\urladdr{https://www.shaunallison.com}

\thanks{The author would like to thank Clinton Conley, Aristotelis Panagiotopoulos, and Assaf Shani for many helpful discussions, and John Clemens and Samuel Coskey for sharing a preprint of their paper.}



\begin{abstract}
    A variation of the Scott analysis of countable structures is applied to actions of non-Archimedean TSI Polish groups acting continuously on a Polish spaces. We give results on the \emph{potential Borel complexity spectrum} of such groups, and define orbit equivalence relations that are universal for each Borel complexity class. We also identify an obstruction to classifiability by actions of such groups, namely generic ergodicity with respect to $E_\infty$, which we apply to Clemens-Coskey jumps of countable Borel equivalence relations. Finally, we characterize the equivalence relations that are both Borel-reducible to $=^+$ and classifiable by non-Archimedean TSI Polish groups, extending a result of Ding and Gao. In the process, several tools are developed in the Borel reducibility theory of orbit equivalence relations which are likely to be of independent interest.
\end{abstract}

\maketitle

\section{Introduction}
The Scott analysis of countable structures can be described as the analysis of continuous actions of $S_\infty$ on Polish spaces, where $S_\infty$ is the Polish group of permutations of $\omega$ with the pointwise convergence topology. A generalization of the Scott analysis to all Polish groups was proposed in \cite{Hjorth_2010}. We will call this the \textbf{Hjorth analysis of Polish group actions}. This work was developed further in \cite{Drucker_2021} to prove some boundedness principles in orbit equivalence relations, and was a central tool used in \cite{Allison_Panagio_2021}. In this paper, we further develop this technique and apply it to an analysis of the actions of non-Archimedean TSI Polish groups.

A Polish group is {\bf non-Archimedean} if it is homeomorphic to a closed subgroup of $S_\infty$, the Polish group of permutations of the natural numbers. A Polish group is {\bf TSI} if it has a metric $d$ which is compatible with the topology and two-sided invariant (i.e. $d(a, b) = d(ca, cb) = d(ac, bc)$ for every $a, b, c$ in the group). Easily, the non-Archimedean TSI Polish groups are exactly the closed subgroups of product groups $\prod_{n=1}^\infty \Delta_n$ for countable discrete groups $\Delta_n$ (see e.g. \cite{Gao_Xuan_2014}). Such groups have a natural tree structure which is reflected in their actions. We study this structure to refine the results of \cite{HKL_1998}, which related potential Borel complexity and Borel reducibility of actions of non-Archimedean Polish groups, to actions of non-Archimedean TSI Polish groups.

Let $\Gamma$ be any pointclass of Borel sets which is closed under continuous preimages such as $\BPi^0_2$, $\BSigma^0_3$, etc. (which we will refer to as a Borel complexity class). An equivalence relation $E$ on a Polish space $\Xspace$ is {\bf potentially} $\Gamma$ if there is a Polish topology $\sigma$ on $\Xspace$, which is compatible with the original topology in the sense that their Borel sets coincide, such that $E$ is in $\Gamma$ as a subset of $\Xspace \times \Xspace$ with the product topology. This gives a notion of complexity of equivalence relations which is robust under compatible changes of topology.

Given two equivalence relations $E$ and $F$ on Polish spaces $\Xspace$ and $\Yspace$, we say that $E$ is {\bf Borel reducible} to $F$, denoted $E \le_B F$, if there exists a Borel function $f : \Xspace \rightarrow \Yspace$ such that $x_1 \mathrel{E} x_2$ if $f(x_1) \mathrel{F} f(x_2)$ for every $x_1, x_2 \in \Xspace$. It's easy to check that if $E \le_B F$ and $F$ is potentially $\Gamma$, then $E$ is potentially $\Gamma$ (recall that one may always pass to a compatible Polish topology in which a given Borel function becomes continuous).

This paper is part of a larger project to understand the potential Borel complexity and Borel reducibility of orbit equivalence relations induced by various classes of non-Archimedean Polish groups (See \cite{HKL_1998}, \cite{Ding_Gao_2017}, \cite{Shani_2021}, and \cite{Clemens_Coskey}).

\subsection{Classification by non-Archimedean TSI actions}

Given a class $\mathcal{C}$ of Polish groups, we say that an equivalence relation $E$ is {\bf classifiable by $\mathcal{C}$-actions} if there is a Polish group $\Gspace$ in $\mathcal{C}$ and a Polish $\Gspace$-space $\Xspace$ such that $E \le_B E^\Gspace_\Xspace$. 
In order to show that an equivalence relation $E$ is not classifiable by $\mathcal{C}$-actions, one often shows that $E$ is generically $E^H_Y$-ergodic for any $H \in \mcC$ and Polish $H$-space $Y$, (i.e. for any Baire-measurable homomorphism $f : \Xspace \rightarrow \Yspace$ there is a comeager subset $C \subseteq \Xspace$ such that $f[C]$ is contained in a single $H$-orbit). If $E$ does not have any comeager classes, then of course this implies that there is no such Borel reduction.

In \cite{Hjorth_2000} and the survey \cite{Kechris_1997}, two different conditions are identified which imply generic ergodicity with respect to actions of non-Archimedean Polish groups. One of them is known as \emph{generic preturbulence}, where preturbulence is a local property stating that some notion of local orbits are somewhere dense. The other is generic ergodicity with respect to a particular simple equivalence relation (namely $=^+$, the Friedman-Stanley jump of equality which we will define later). 

We show that in a similar vein, generic ergodicity with respect to $E_\infty$ is equivalent to generic ergodicity with respect to actions of non-Archimedean TSI Polish groups.

\begin{restatable*}{theorem}{topobstructionthm}\label{th:top_obstruction}
An equivalence relation $E$ is generically ergodic with respect to any orbit equivalence relation $E^\Gspace_\Xspace$ with $\Gspace$ non-Archimedean TSI Polish if and only if $E$ is generically $E_\infty$-ergodic.
\end{restatable*}

Recall that the equivalence relation $E_\infty$ is defined as the shift action of $F_2$, the free group on two-generators, on $2^{F_2}$. The property of $E_\infty$ we will be exploiting is the fact that it is the maximal countable Borel equivalence relation under Borel reducibility (see e.g.  \cite[Theorem 7.3.8]{Gao_2009}).

In \cite{Clemens_Coskey}, Clemens and Coskey define a jump operator $E \mapsto E^{[\Delta]}$ on equivalence relations for every countable group $\Delta$, and show that in the case that $E$ is countable it is classified by a non-Archimedean CLI (complete left-invariant) Polish group. As it was not known if classification by CLI Polish groups implied classification by TSI Polish groups, they asked if these relations are classified by TSI Polish groups. We can apply Theorem \ref{th:top_obstruction} to conclude that many of these equivalence relations, such as $E_0^{[\mathbb{Z}]}$, are not classifiable by non-Archimedean TSI Polish groups.

\begin{restatable*}{corollary}{ccapplication}\label{cor:cc_application}
If $E$ is generically ergodic with every $E$-class meager, and $\Delta$ is infinite, then $E^{[\Delta]}$ is not classifiable by non-Archimedean TSI Polish groups.
\end{restatable*}

We note that generic ergodicity with respect to $=_{\mathbb{R}}$, equality on the reals, is usually just referred to as generic ergodicity.

While a local property similar to turbulence but for non-Archimedean TSI Polish groups has not yet been identified, a global topological property known as \emph{generic semi-connectnedness of the balanced graph} is identified in \cite{Allison_Panagio_2021}, which implies generic ergodicity with respect to actions of all TSI Polish groups. It is not clear that there is a converse result for the property identified in \cite{Allison_Panagio_2021}, and no dichotomy has yet been found along the lines of \cite{Hjorth_2002} but for classification by $\mcC$-actions, where $\mcC$ is the class of non-Archimedean TSI Polish groups rather than the entire class of non-Archimedean Polish groups.

By replacing $E_\infty$ in the previous discussion with $E_0$, we isolate another anticlassification result. Recall that $E_0$ is the equivalence relation $2^\omega$ which relates two elements $x$ and $y$ if they eventually agree, i.e. $x(n) = y(n)$ for large enough $n$. By Slaman-Steel, this is Borel-bireducible with the orbit equivalence relation of the shift action of $\mathbb{Z}$ on $2^\mathbb{Z}$. If we follow the same line of argument of Theorem \ref{th:top_obstruction} with respect to $E_0$-ergodicity, we get the following result:
\begin{restatable*}{theorem}{topobstructionthmtwo}\label{th:top_obstruction2}
An equivalence relation $E$ is generically ergodic with respect to any orbit equivalence relation $E^\Gspace_\Xspace$ with $\Gspace$ non-Archimedean abelian Polish if and only if $E$ is generically $E_0$-ergodic.
\end{restatable*}

\subsection{Universal Actions}
In \cite{HKL_1998}, a thorough analysis was performed for actions of non-Archimedean Polish groups with Borel orbit equivalence relations, completely characterizing the possible potential Borel complexity classes. They produced the following list:

\begin{equation}\label{eq:complexity_spectrum}
\BDelta^0_1, \BPi^0_1, \BSigma^0_2, \BPi^0_n, D(\BPi^0_n) \text{ ($n \ge 3$)}, \bigoplus_{\alpha < \lambda} \BPi^0_\alpha, \BPi^0_\lambda, \BSigma^0_{\lambda+1}, \BPi^0_{\lambda+n}, D(\BPi^0_{\lambda+n}) \text{ ($n \ge 2$, $\lambda$ limit)}. \tag{$*$}
\end{equation}
Here, an equivalence relation $E$ is $D(\BPi^0_\alpha)$ if it is the difference of two $\BPi^0_\alpha$ sets, and $\bigoplus_{\alpha < \lambda} \BPi^0_\alpha$ if there is a partition $\Xspace = \bigsqcup_{n \in \omega} \Xspace_n$ of the domain of $E$ into Borel sets such that for every $n$, there is some $\alpha < \lambda$ such that $E \cap (X_n \times X_n)$ is $\BPi^0_\alpha$.

It's straightforward to check that this list is linearly-ordered in the natural way. For example, starting from the beginning, we have:
\[\BDelta^0_1 \subseteq \BPi^0_1 \subseteq \BSigma^0_2 \subseteq \BPi^0_3 \subseteq D(\BPi^0_3) \subseteq \BPi^0_4 \subseteq D(\BPi^0_4) \subseteq ... \]

We can also analyze classification by $\mathcal{C}$-actions of a given Borel complexity. Given a complexity $\Gamma$ from $(*)$, we say that an equivalence relation $E$ is {\bf classifiable by $\mathcal{C}$-actions of complexity $\Gamma$} if there is some Polish group $\Gspace$ from $\mathcal{C}$ and a Polish $\Gspace$-space $\Xspace$ such that $E \le_B E^\Gspace_\Xspace$ and $E^G_X$ has complexity $\Gamma$.

In \cite[Theorem 6.1]{Ding_Gao_2017}, Ding and Gao proved that if $E$ is an essentially countable equivalence relation and $E \le_B E^\Gspace_\Xspace$ where $\Gspace$ is non-Archimedean abelian Polish, then $E \le_B E_0$.

We will be able to generalize this result by characterizing the equivalence relations classifiable by $\mathcal{C}$-actions of potential complexity $\BPi^0_3$ when $\mathcal{C}$ is the class of non-Archimedean TSI Polish groups, or the non-Archimedean abelian Polish groups. Recall that given equivalence relations $E_n$ for $n \in \omega$ on Polish spaces $X_n$, the product $\prod_n E_n$ is the equivalence relation on the product Polish space $\prod_n X_n$ relating sequences $(x_n)_{n \in \omega}$ and $(y_n)_{n \in \omega}$ iff $x_n E_n y_n$ for every $n$.

\begin{restatable*}{corollary}{classificationcor}\label{cor:classification}
Suppose $E \le_B F$ where $F$ is $\BPi^0_3$ and idealistic. If $E$ is also classifiable by a non-Archimedean TSI action, then $E \le_B E_\infty^\omega$. If $E$ is classifiable by a non-Archimedean abelian action, then $E \le_B E_0^\omega$.
\end{restatable*}

In addition to characterizing the potential complexity spectrum of $S_\infty$, the authors of \cite{HKL_1998} defined, for each complexity class in (*), a single action of $S_\infty$ belonging to that class, which is also universal for actions of $S_\infty$ belonging to that class. Given a Polish group $\Gspace$ and complexity class $\Gamma$, we say that an equivalence relation $E$ is {\bf universal for potentially $\Gamma$ equivalence relations classifiable by $\Gspace$} if $E$ is induced by a continuous action of $\Gspace$ and is potentially $\Gamma$, and for any other such equivalence relation $F$, we have $F \le_B E$. In other words, they showed that for every complexity class $\Gamma$ in $(*)$, there is an equivalence relation which is universal for orbit equivalence relations induced by $S_\infty$ of potential complexity $\Gamma$.

The previous lemma shows that $E_\infty^\omega$ and $E_0^\omega$ are universal among potentially $\BPi^0_3$ orbit equivalence relations classifiable by non-Archimedean TSI and abelian Polish groups respectively. We show that such equivalence relations exist for every possible potential complexity class.

\begin{restatable*}{theorem}{tsiuniversalactionsthm}\label{th:tsi_universal_actions}
For every complexity class $\Gamma$ in $(*)$ and for $\mcC$ equal to the non-Archimedean TSI Polish groups or the non-Archimedean abelian Polish groups, there is an equivalence relation that is universal for potentially $\Gamma$ equivalence relations classifiable by $\mcC$ groups.
\end{restatable*}

We can generalize this definition in the natural way. Let $\mathcal{C}$ be a class of Polish groups, such as the non-Archimedean TSI Polish groups. We say that an equvivalence relation $E$ is {\bf universal for potentially $\Gamma$ orbit equivalence relations induced by $\mathcal{C}$-actions} if $E$ is potentially $\Gamma$, induced by a continuous action of some group in $\mathcal{C}$, and for any such equivalence relation $F$, we have $F \le_B E$. In \cite{HKL_1998}, it is shown for every complexity $\Gamma$ in $(*)$ that there is an equivalence relation that is universal for potentially $\Gamma$ orbit equivalence relations induced by actions of non-Archimedean Polish groups. We can show that this is also true when $\mathcal{C}$ is the class of non-Archimedean TSI and non-Archimedean abelian Polish groups.

\begin{restatable*}{corollary}{tsiuniversalactionscor}\label{cor:tsi_universal_actions}
For $\mathcal{C}$ equal to the non-Archimedean abelian Polish groups, or the non-Archimedean TSI Polish groups, and for any $\Gamma$ in $(*)$, there is an equivalence relation that is universal for potentially $\Gamma$ orbit equivalence relations induced by $\mathcal{C}$-actions.
\end{restatable*}

\subsection{The Potential Complexity Spectrum}

The authors of \cite{HKL_1998} proved that $(*)$ was exhaustive for actions of non-Archimedean Polish groups with Borel equivalence relations. For example, they showed that if such an equivalence relation $E$ is potentially $\BSigma^0_3$, then it is potentially $\BSigma^0_2$, and if it is potentially $\BSigma^0_4$ then it is potentially $D(\BPi^0_3)$.

Given a Polish group $G$, a Polish $G$-space $(X, a)$ consists of a Polish space $X$ and a continuous action $a : G \curvearrowright X$ (though we often write $(X, a)$ as simply $X$ when there is no ambiguity). For a non-Archimedean Polish group $G$, we will let the {\bf potential Borel complexity spectrum} of $G$ to be exactly the complexity classes $\Gamma$ in $(*)$ for which there is a Polish $G$-space $X$ such that $E^G_X$ is potentially $\Gamma$, but not potentially $\Gamma'$ for any complexity class $\Gamma'$ appearing earlier in the list.

This raises a family of interesting questions -- given a non-Archimedean Polish group $\Gspace$, what is its potential Borel complexity spectrum? By exploiting the natural tree structure of the non-Archimedean TSI Polish groups, we can give the following partial result.

\begin{restatable*}{theorem}{tsicomplexityspectrumthm}\label{th:tsi_complexity_spectrum}
For any countable discrete group $\Delta$, the potential Borel complexity spectrum of the group $\Delta^\omega$ is an initial segment of $(*)$.
\end{restatable*}

More generally, we actually prove this for groups $G^\omega$ where $G$ is non-Archimedean TSI. From this we can conclude that if $G$ is a non-Archimedean TSI Polish group, and $G^\omega$ induces Borel orbit equivalence relations of arbitrarily high complexity, then every complexity class in $(*)$ is in the potential Borel complexity spectrum of $G^\omega$. In this case, we say that this group has {\bf full potential Borel complexity spectrum}. It was proven in \cite{HKL_1998} that $S_\infty$ has full potential Borel complexity spectrum. Using Solecki's results on tame abelian product groups (see \cite{Solecki_1995}), we can give more examples of such groups.

\begin{restatable*}{corollary}{tsicomplexityspectrumcor}\label{cor:tsi_complexity_spectrum}
The groups $\mathbb{Z}^\omega$, $(\mathbb{Z}_p^{<\omega})^\omega$ for prime $p$, and $F_2^\omega$ all have full potential Borel complexity spectrum.
\end{restatable*}

\subsection{Structure of the Paper}

In Section 2, we cover some basic definitions and results about Borel reductions and changes of topology that we will need throughout the paper, some of which is new or previously unpublished. In Section 3, we review Hjorth's Scott-like analysis of Polish group actions.  In Section 4, we compute the Hjorth analysis for the special case of actions of non-Archimedean TSI Polish groups, and define a family of universal equivalence relations for every Borel complexity class. In Section 5, we prove all of the results stated in the introduction.

\section{Potential Borel complexity of orbit equivalence relations}

Given a Polish group $\Gspace$ and Polish $\Gspace$-space $\Xspace$ we recall that for an arbitrary subset $A \subseteq \Xspace$ and nonempty open $\mcV \subseteq \Gspace$ the sets $A^{\Delta \mcV}$ and $A^{*\mcV}$, called the {\bf Vaught transforms} of $A$, are defined by:
\[A^{\Delta \mcV} := \{x \in \Xspace \mid \text{for a nonmeager set of } g \in \mcV, \; g \cdot x \in A\}\]
and
\[\quad A^{* \mcV} := \{x \in \Xspace \mid \text{for a comeager set of } g \in \mcV, \; g \cdot x \in A\} .\]
See \cite[Lemma 5.1.7]{Becker_Kechris_1996} for some basic properties of Vaught transforms.

\subsection{Changes of topology}

Given a Polish group $G$ and a Polish $G$-space $X$ such that $E^G_X$ has potential Borel complexity $\Gamma$, by definition there is a Polish topology $\sigma$ on $X$ such that $E^G_X$ is $\Gamma$ in $X \times X$ in the product topology. However, it is often useful to choose such $\sigma$ where the action of $G$ is continuous with respect to $\sigma$. We will show how to do this.





\begin{proposition}[Hjorth]\label{pr:hjorth_topology_change}
Let $\alpha < \omega_1$ and let $\{A_n \mid n < \omega\}$ be a countable family of $\BSigma^0_\alpha$ subsets of $\Xspace$. Then there is a finer topology $\sigma$ such that for every $A_n$ and open $\mcV \subseteq \Gspace$, $A_n^{\Delta \mcV}$ is open in $\sigma$. Furthermore, every set $U$ open in $\sigma$ is $\BSigma^0_{\alpha}$ in the original topology.
\end{proposition}

\begin{proof}
See \cite[4.3.3]{Gao_2009}.
\end{proof}

\begin{lemma}\label{lm:top_change}
Let $\Gspace$ be a Polish group acting on a set $X$ and suppose $\tau$ and $\sigma$ are two different topologies on $X$ making it a Polish $G$-space. Assume that $\mcU^{\Delta \mcV}$ is open in $\sigma$ whenever $\mcU \subseteq \Xspace$ is open in $\tau$ and $V \subseteq G$ open.

Then for any open $\mcV \subseteq \Gspace$, any ordinal $\alpha \ge 1$, and any set $A \subseteq \Xspace$,
\[\text{if } A \text{ is } \BPi^0_\alpha \text{ in } \tau, \quad \text{then} \quad A^{*\mcV} \text{ is } \BPi^0_\alpha \text{ in } \sigma,\]
and
\[\text{if } A \text{ is } \BSigma^0_\alpha \text{ in } \tau, \quad \text{then} \quad A^{\Delta \mcV} \text{ is } \BSigma^0_\alpha \text{ in } \sigma.\]
\end{lemma}

\begin{proof}
This is true for $\alpha = 0$ by the assumption. We prove the rest by induction.

Fix $\alpha$ and assume the claim is true for all $\beta < \alpha$. If $A$ is $\BSigma^0_\alpha$ in $\tau$, then we can write $A = \bigcup_{n \in \omega} A_n$ where each $A_n$ is $\BPi^0_{\beta_n}$ for some $n < \omega$. Then
\[A^{\Delta \mcV} = \Big[\bigcup_{n \in \omega} A_n\Big]^{\Delta \mcV} = \bigcup_{n \in \omega} A_n^{\Delta \mcV} = \bigcup_{n \in \omega} \bigcup_{m \in \omega} A_n^{*\mcW_m}\]
by basic properties of the Vaught transforms, where $W_m$ enumerates a basis of open sumsets of $V$.
By the induction hypothesis, each $A_n^{*\mcW}$ is $\BPi^0_{\beta_n}$ in $\sigma$ and thus $A^{\Delta \mcV}$ is $\BPi^0_\alpha$ in $\sigma$.

On the other hand, if $A$ is $\BSigma^0_\alpha$ then its complement is $\BPi^0_\alpha$ and we just proved that $X \setminus A^{\Delta V} = (X \setminus A)^{*V}$ is $\BPi^0_\alpha$ in $\sigma$. Thus $A^{\Delta V}$ is $\BSigma^0_\alpha$ in $\sigma$. 
\end{proof}

This allows us to give the promised equivalent statement of potential Borel complexity for orbit equivalence relations.

\begin{theorem}\label{th:top_change}
Let $\Gspace$ be a Polish group and $\Xspace$ a Polish $\Gspace$-space. If $E^\Gspace_\Xspace$ is potentially $\Gamma$ for some $\Gamma$ in $(*)$, then there is a topology $\sigma$ on $\Xspace$ such that $(\Xspace, \sigma)$ is a Polish $\Gspace$-space and $E^\Gspace_\Xspace$ is $\Gamma$.
\end{theorem}

\begin{proof}
Let $\tau$ be a Polish topology on $\Xspace$ such that $E^G_X$ is $\Gamma$ in $\tau$. By Proposition  \ref{pr:hjorth_topology_change}, we can find a topology $\sigma$ such that $(\Xspace, \sigma)$ is a Polish $\Gspace$-space, and $\mcU^{\Delta \mcV}$ is open in $\sigma$ for every basic open $\mcV \subseteq \Gspace$ and every basic $\tau$-open $\mcU \subseteq \Xspace$.

Observe that
\[\{ \mcV_1 \times \mcV_2 \mid \mcV_1, \mcV_2 \; \text{basic open}\}\]
is a countable basis for $G \times G$ with the product topology, and futhermore, if $\mcU_1 \times \mcU_2$ is open in $\tau \times \tau$, then by Kuratowski-Ulam,
\[(\mcU_1 \times \mcU_2)^{\Delta (\mcV_1 \times \mcV_2)} = \mcU_1^{\Delta \mcV_1} \times \mcU_2^{\Delta \mcV_2}. \]
Thus the natural product action of $G \times G$ on $X \times X$ and the product topologies $\tau \times \tau$ and $\sigma \times \sigma$ satisfy the conditions of Lemma \ref{lm:top_change}. By noting that $E^G_X = (E^G_X)^{\Delta G \times G} = (E^G_X)^{* G \times G}$, it is easy to see for each $\Gamma$ that it follows from the lemma that $E^G_X$ is $\Gamma$ in $\sigma$.
\end{proof}

\subsection{Actions of non-Archimedean Polish groups}

We recall results of Hjorth-Kechris-Hjorth characterizing the potential Borel complexity of actions of non-Archimedean Polish groups. The results in this paper are essentially an extension of these results to actions of the non-Archimedean TSI Polish groups.

\begin{definition}
Given an equivalence relation $E$ on a Polish space $X$, the \emph{Friedman-Stanley jump} of $E$ is the equivalence relation $E^+$ on $X^\omega$ such that $(x_i)_{i \in \omega} \mathrel{E^+} (y_i)_{i \in \omega}$ iff $\{[x_i]_E \mid i \in \omega\} = \{[y_i]_E \mid i \in \omega\}$.
\end{definition}

Let $\mcC$ denote Cantor space. Recursively define the $\alpha$-iterated power set $\mcP^\alpha(\mcC)$ of $\mcC$ by letting $\mcP^0(\mcC) := \mcC$, and $\mcP^{\alpha+1}(\mcC) := \mcP(\mcP^\alpha(\mcC))$, and for limit $\gamma$ define $\mcP^\lambda(\mcC) := \prod_{\alpha < \gamma} \mcP^\alpha(\mcC)$ (i.e., the set of sequences $(x_\alpha)_{\alpha<\lambda}$ such that each $x_\alpha \in \mcP^\alpha(\mcC)$). The Friedman-Stanley jumps of equality on $\mcC$, denoted $=_\mcC^{+\alpha}$, give a family of equivalence relations of increasing complexity which have elements in $\mcP^\alpha(\mcC)$ as invariants.

For countable ordinals $\alpha$, the relation $=^{+\alpha}_\mcC$ called the \emph{$\alpha$th-iterated Friedman-Stanley jump of equality}, is defined recursively on $\alpha$. The 0th jump $=^{+0}_\mcC$, or just $=_{\mcC}$, is defined to be $= \upharpoonright \mcC \times \mcC$ (``equality on the reals"), $=^{+(\alpha+1)}_\mcC$ is the Friedman-Stanley jump of $=^{+\alpha}_\mcC$, and for limit $\alpha$, $=^{+\alpha}_\mcC$ is the product of $=^{+\beta}_\mcC$ for $\beta < \alpha$. Recall that given a sequence $(E_\beta)_{\beta < \alpha}$ of equivalence relations living on Polish spaces $X_\beta$ for $\beta < \alpha$, the product $\prod_{\beta < \alpha} E_\beta$ is the equivalence relation living on the product Polish space $\prod_{\beta < \alpha} X_\beta$ which relates sequences $(x_\beta)_{\beta < \alpha}$ and $(y_\beta)_{\beta < \alpha}$ iff $x_\beta \mathrel{E_\beta} y_\beta$ for every $\beta < \alpha$. From now on, for notational tidiness, we write the $\alpha$th Friedman-Stanley jump of equality on the reals as $=^{+\alpha}$, dropping the subscript.

The following fundamental result shows the relationship between the Friedman-Stanley jumps of equality and potential Borel complexity of actions of non-Archimedean Polish groups.

\begin{proposition}[Hjorth-Kechris-Louveau]\label{pr:hkl}
Suppose $G$ is a non-Archimedean Polish group and $X$ a Polish $G$-space. Then
\begin{enumerate}
    \item $E^G_X \le_B \; =_{\mcC}$ iff $E^G_X$ is potentially $\BPi^0_1$;
    \item For every $n \in \omega$, $E^G_X \le_B \; =^{+(n+1)}$ iff $E^G_X$ is potentially $\BPi^0_{n+3}$;
    \item For every limit $\lambda$, $E^G_X \le_B \; =^{+\lambda}$ iff $E^G_X$ is potentially $\BPi^0_\lambda$; and
    \item For every limit $\lambda$ and every $n \in \omega$, $E^G_X \le_B \; =^{+(\lambda+n+1)}$ iff $E^G_X$ is potentially $\BPi^0_{\lambda+n+2}$.
\end{enumerate}
Moreover, every $=^{+\alpha}$ is induced by a Borel action of $S_\infty$ on a Polish space.
\end{proposition}

\begin{proof}
See \cite[Theorem 2]{HKL_1998} for the first part and \cite[Section 1]{HKL_1998} for the moreover part.
\end{proof}

We will prove an analogue of this result for the non-Archimedean TSI Polish groups, and much of the same techniques will be employed here, albeit in a more refined form.

\subsection{Inverses of Borel reductions}
Given a Polish space $X$ and any subset $A \subseteq X$, we say that a set $B \subseteq X$ is \emph{Borel relative to $A$} if there is a Borel set $B_0 \subseteq X$ such that $B_0 \cap A = B \cap A$. This is the same as saying that $B \cap A$ is Borel in the subspace topology on $A$.

\begin{definition}
Suppose $X$ and $Y$ are Polish spaces and $A \subseteq X$ any subset. A function $\Phi : A \rightarrow \mathcal{P}(\mathcal{P}(Y))$ is \emph{Borel-on-Borel relative to $A$} if for every Polish space $Z$ and Borel $B \subseteq Z \times X \times Y$, the set $\{(z, x) \in Z \times A \mid B_{z, x} \in \Phi(x)\}$ is Borel relative to $Z \times A$. We simply say that such $\Phi$ is Borel-on-Borel when it is Borel-on-Borel relative to the entire space $X$.
\end{definition}

We will need a slightly stronger version of the standard ``large-section" uniformization result (see \cite[Theorem 18.6]{Kechris_1997}) for an analytic domain. The same proof works with only minor modifications as is mentioned in \cite{Kechris_CDST_corrections}. For completeness, we include a proof.

\begin{proposition}\label{pr:large_section}
Suppose $P \subseteq X \times Y$ is Borel for Polish spaces $X$ and $Y$, and for every $x$, there is a $\sigma$-ideal $I_x$ on $Y$ such that $P_x \not\in I_x$ whenever $P_x$ is nonempty. Assume also that the map $x \rightarrow I_x$ is Borel-on-Borel relative to $\Proj_X(P)$. Then there is a Borel uniformization of $P$.
\end{proposition}

\begin{proof}
Fix a Borel injection $f : \omega^\omega \rightarrow X \times Y$ such that $f[\omega^\omega] = P$.
For $s \in \omega^{<\omega}$, fix some Borel set $B_s \subseteq X$ such that
\[x \in B_s \quad \text{iff} \quad (f[N_s])_x \not\in I_x\]
for every $x \in \Proj_X(P)$.
Define a Borel map $x \mapsto T_x$ from $X$ to the (Polish) space of trees on $\omega$ such that 
\[s \in T_x \quad \text{iff} \quad x \in B_s\]
for every $s \in \omega^{<\omega}$.
Define $D \subseteq X$ to be the set of $x \in X$ such that $T_x$ codes a nonempty pruned tree. Then $D$ is Borel and $D \supseteq \Proj_X(P)$ and we can define a Borel map $h : D \rightarrow \omega^\omega$ sending $x$ to the left-most path in the tree coded by $T_x$. Then the map from $x$ to the unique $y$ such that $f(h(x)) = (x, y)$ is a Borel map from $D$ to $Y$ which restricts to a Borel uniformization of $P$.
\end{proof}

We recall the following definition.

\begin{definition}
An equivalence relation $E$ on a Polish space $X$ is \emph{idealistic} if there is a $E$-invariant, Borel-on-Borel map $x \mapsto I_{[x]_E}$ such that each $I_{[x]_E}$ is a $\sigma$-ideal on $X$ where $[x]_E \not\in I_{[x]_E}$.
\end{definition}

Note that equivalence relations induced by Borel actions of Polish groups are idealistic (see \cite[5.4.10]{Gao_2009}). 

The following theorem is similar to \cite[Theorem 4.1]{Gao_2001} and \cite[Lemma 3.7]{Kechris_Macdonald_2016}, however we need this stronger statement 

\begin{theorem}\label{th:borel_reduction_inverse}
Suppose $E_1$, $E_2$, and $F$ are equivalence relations on Polish spaces $X_1$, $X_2$, and $Y$ respectively, with $E_1$ analytic and $E_2$ Borel and idealistic. Let $f : Y \rightarrow X_1$ be a Borel reduction of $F$ to $E_1$ and assume also $F \le_B E_2$. Then there is a $E_1$-invariant Borel set $C \subseteq X_1$ such that $f[Y] \subseteq C$ and $E_1 \upharpoonright C \le_B E_2$.
\end{theorem}

\begin{proof}
Fix a Borel reduction $g : Y \rightarrow X_2$ from $F$ to $E_2$, and let $\Phi \subseteq \mathcal{P}(X_1 \times X_2)$ such that 
\[P \in \Phi \quad \text{iff} \quad \forall (x_1, x_2), (y_1, y_2) \in P, \; (x_1 E y_1 \rightarrow x_2 E_2 y_2).\] 

Define $P_0 \subseteq X_1 \times X_2$ by
\[P_0 := \{(x_1, x_2) \mid \exists y \in Y, f(y) E_1 x_1 \; \text{and} \; g(y) E_2 x_2\}. \]
Because $\Phi$ is $\BPi^1_1$-on-$\BSigma^1_1$ and $P_0\in \Phi$ is analytic, we can find some Borel $P_1 \supseteq P_0$ such that $P_1 \in \Phi$ by the First Reflection Theorem (see \cite[35.10]{Kechris_1997}). 

Fix a Borel-on-Borel map $x_2 \mapsto I_{[x_2]_{E_2}}$ witnessing the fact that $E_2$ is idealistic. For $x_1 \in \Proj_{X_1}(P_1)$, define an ideal $J_{x_1}$ on $X_2$ where 
\[A \in J_{x_1}\quad \text{iff} \quad \exists x_2, \; (x_1, x_2) \in P_1 \;  \text{and} \; A \in I_{[x_2]_{E_2}}. \]
Define
\[P := \{(x_1, x_2) \in P_1 \mid (P_1)_{x_1} \not\in J_{x_1}\} = \{(x_1, x_2) \in P_1 \mid (P_1)_{x_1} \not\in I_{[x_2]_{E_2}}\},\]
which is a Borel set, where the equality follows from the fact that $P_1 \in \Phi$.

We claim that the map $x \mapsto J_{x}$ is Borel-on-Borel relative to $\Proj_{X_1}(P)$. To see this, fix some Polish space $Z$ and a Borel subset $B \subseteq Z \times X_1 \times X_2$. Let
\[A_1 := \{(z, x_1) \in P \mid \exists x_2 \in X_2, (x_1, x_2) \in P \; \text{and} \; B_{z, x_1} \in I_{[x_2]_{E_2}}\}\]
and
\[A_2 := \{(z, x_1) \in Z \times X_1 \mid \forall x_2 \in X_2, \; ((x_1, x_2) \in P  \rightarrow B_{z, x_1} \in I_{[x_2]_{E_2}})\}.\]
By analytic separation, fix some Borel set $A \supseteq A_1$ such that $A \cap A_2 = \emptyset$. Observe that 
\[A_1 \cap (Z \times \Proj_{X_1}(P)) = A_2 \cap (Z \times \Proj_{X_1}(P)) = A \cap (Z \times \Proj_{X_1}(P)) \]
and so $A_1$ is Borel relative to $\Proj_{X_1}(P)$.

By Proposition \ref{pr:large_section}, we can fix some Borel uniformization $h : \Proj_{X_1}(P) \rightarrow X_2$ of $P$. In particular, $B_0 := \Proj_{X_1}(P)$ is Borel. By the definition of $\Phi$, we have that $h$ is a reduction of $E_1 \upharpoonright B_0$ to $E_2$. Finally, we define $B := B_0 \cap B_0^{*G}$, which is Borel by basic properties of the Vaught transform. This is $G$-invariant and contains $f[Y]$ because $P_0 \subseteq P$. Thus $h \upharpoonright B$ is our desired reduction.
\end{proof}


We note that the assumption that $E_2$ is idealistic appears to be necessary. For example, if $F$ is a smooth equivalence relation and $E_1$ is equality on the reals, and $E_2 = F$, then the conclusion would imply the existence of a Borel transversal of $F$. However a smooth equivalence relation which is not idealistic need not have a Borel transversal. Given any closed subset $C \subseteq Z_1 \times Z_2$, one can define the relation $E_C$ on $C$ where $(z_1, z_2) E_C (z_1', z_2')$ iff $z_1 = z_1'$. This is clearly smooth as witnessed by the reduction $(z_1, z_2) \mapsto z_1$. The Borel transversals of $E_C$ are exactly the Borel uniformizations of $C$, which need not exist in general.

The previous result allows us to prove the following generalization of Friedman's theorem, which was originally proved in the case that $G$ is $S_\infty$.

\begin{corollary}\label{cor:upper_bound}
Suppose $E$ is a Borel equivalence relation which is classifiable by a non-Archimedean Polish group $G$. Then $E \le E^G_X$ for some Polish $G$-space $X$ with $E^G_X$ Borel.
\end{corollary}

\begin{proof}
Fix a Polish $G$-space $X$ such that $E \le_B E^G_X$. Because $G$ is a closed subgroup of $S_\infty$, by \cite[Theorem 2.3.5]{Becker_Kechris_1996} there is a Polish $S_\infty$-space $Y$ such that $E^G_X \sim_B E^{S_\infty}_Y$. By a theorem of Friedman, see \cite[Theorem 1.5]{Friedman_2000}, we can find another Polish $S_\infty$-space $Z$ such that $E \le_B E^{S_\infty}_Z$ and $E^{S_\infty}_Z$ is Borel.

By Theorem \ref{th:borel_reduction_inverse}, we can find an invariant Borel set $C \subseteq X$ such that $E^G_X \upharpoonright C \le_B E^{S_\infty}_Z$. By Proposition \ref{pr:hjorth_topology_change}, we can find a compatible Polish topology in which $C$ is closed, and in particular is a Polish $G$-space in the relative topology. Thus $E \le E^G_C$ with $E^G_C$ Borel.
\end{proof}

Finally, assuming $\mcC$ is a class of non-Archimedean Polish groups, we observe that in order to show that a Borel equivalence relation $E$ is classifiable by $\mcC$-actions of potential complexity $\Gamma$, it is (to some degree) enough to show these two properties separately.

It is obvious from previous remarks that given a Borel complexity class $\Gamma$, an equivalence relation $E$ is potentially $\Gamma$ iff it is Borel-reducible to an equivalence relation which has exact complexity $\Gamma$. We say that $E$ has the (seemingly stronger) property of being \textbf{idealistically $\Gamma$} iff there is an idealistic equivalence relation $E_I$ which has exact complexity $\Gamma$, and $E \le_B E_I$. Under this assumption the proof goes through.

\begin{corollary}\label{cor:upper_bound2}
Suppose $E$ is an equivalence relation which is idealistically $\Gamma$ and classifiable by a non-Archimedean Polish group $G$. Then $E$ is classifiable by $G$-actions of complexity $\Gamma$.
\end{corollary}

\begin{proof}
Fix a Polish $G$-space $X$ such that $E \le_B E^G_X$. By Theorem \ref{th:borel_reduction_inverse}, we can find an invariant Borel set $C \subseteq X$ such that $E^G_X \upharpoonright C \le_B E^{S_\infty}_Z$. By Proposition \ref{pr:hjorth_topology_change}, we can find a compatible Polish topology in which $C$ is closed, and in particular is a Polish $G$-space in the relative topology. Finally, because $E^{S_\infty}_Z$ is potentially $\Gamma$ then so is $E^G_C$, and thus by Theorem \ref{th:top_change} there is another compatible Polish topology in which $C$ is still a Polish $G$-space and $E^G_C$ has exact complexity $\Gamma$.
\end{proof}

One would hope that the assumption can be weakened to $E$ being potentially $\Gamma$ instead of just idealistically $\Gamma$. 

\begin{question}
If an equivalence relation $E$ is potentially $\Gamma$ and Borel-reducible to a Borel idealistic equivalence relation, must $E$ be idealistically $\Gamma$?
\end{question}

This appears to be related to the deep Kechris-Louveau conjecture regarding equivalence relations which are not above the universal hypersmooth equivalence relation $E_1$ with respect to $\le_B$ \cite{Kechris_Louveau_1997}.

\section{The Hjorth analysis of Polish group actions}

In \cite{Hjorth_2010}, Hjorth outlined a method of extending the Scott analysis of isomorphism of countable structures to arbitrary Polish group actions. That work was continued in \cite{Drucker_2021}. We refer to this as the {\bf Hjorth analysis of Polish group actions}. While we will primarily be concerned with actions of non-Archimedean Polish groups, which can be done within the context of the conventional Scott analysis, it takes just as much work to describe the more general analysis, and it is more conventient to develop the more general theory for future work into potential Borel complexity.

For the rest of the section, let $\Gspace$ be a Polish group and $\Xspace$ a Polish $\Gspace$-space.

\begin{definition}[Hjorth] For open $U, V \subseteq G$ and $x, y \in X$, we say that
$(x, \mcU) \le_1 (y, \mcV)$ iff $\mcU \cdot x \subseteq \overline{\mcV \cdot y}$.

For $\alpha \ge 1$, $(x, \mcU) \le_{\alpha+1} (y, \mcV)$ iff for every nonempty open $\mcU' \subseteq \mcU$, there is some nonempty open $\mcV' \subseteq \mcV$ such that $(y, \mcV') \le_\alpha (x, \mcU')$.

Finally, for limit $\lambda$, $(x, \mcU) \le_\lambda (y, \mcV)$ iff for every $1 \le \alpha < \lambda$, $(x, \mcU) \le_\alpha (y, \mcU)$.
\end{definition}

\begin{proposition}[Hjorth]\label{pr:hjorth_inv}
The relation $(x, \mcU) \le_\alpha (y, \mcV)$ holds iff for every $\BPi^0_\alpha$ set $A$, $y \in A^{* \mcV} \Rightarrow x \in A^{* \mcU}$.
\end{proposition}

\begin{proof}
For the base case, observe that
\begin{align*}
&(x, U) \le_1 (y, V)\\
\iff &\overline{Ux} \subseteq \overline{Vy} \\
\iff &\forall \; \text{open} \; O, \; O \cap (U \cdot x) \neq \emptyset \Rightarrow O \cap (V \cdot y) \neq \emptyset\\
\iff &\forall \; \text{open} \; O, \; x \in O^{\Delta U} \Rightarrow  y \in O^{\Delta V}\\
\iff &\forall \; \text{closed} \; C, \; y \in C^{* V} \Rightarrow  x \in C^{* U}.
\end{align*}

For limit $\alpha$, assuming the claim is true for all $\beta < \alpha$, we have that
\begin{align*}
    &(x, U) \le_\alpha (y, V) \\
    \iff &\forall \beta < \alpha, \; (x, U) \le_\beta (y, V) \\
    \iff &\forall \beta < \alpha, \; \forall \; \BPi^0_\beta \; \text{set} \; A, \; y \in A^{*V} \Rightarrow x \in A^{*U}\\
    \iff &\forall \; \BPi^0_\alpha \; \text{set} \; A, \; y \in A^{*V} \Rightarrow x \in A^{*U}.
\end{align*}
The last equivalence follows from the fact that the $\BPi^0_\alpha$ sets $A$ are exactly the countable intersections $\bigcap_n A_n$ where for each $A_n$ is $\BPi^0_{\beta_n}$ for some $\beta_n < \alpha$, and $A^{*W} = (\bigcap_n A_n)^{*W} = \bigcap_n A_n^{*W}$ for every open $W \subseteq G$.

For the successor case $\alpha = \beta+1$, assuming the claim is true for $\beta$, we argue each direction separately. We first observe that by the induction hypothesis we have that $(x, U) \le_\alpha (y, V)$ iff for every nonempty open $U' \subseteq U$ there exists some nonempty $V' \subseteq V$ such that for every $\BPi^0_\beta$ set $A$, $x \in A^{*U'}$ implies $y \in A^{*V'}$.

Suppose $(x, U) \le_\alpha (y, V)$ and $A$ is a $\BPi^0_\alpha$ set such that $x \not\in A^{*U}$. We will show that $y \not\in A^{*U}$. Write $A = \bigcap_n A_n$ for $\BSigma^0_\beta$ sets $A_n$. Then we can fix some nonempty open $U' \subseteq U$ and some $n$ such that $x \in (X 
\setminus A_n)^{* U'}$. By the assumption find some nonempty open $V' \subseteq V$ such that $(y, V') \le_\beta (x, U')$ in which case $y \in (X \setminus A_n)^{*U'}$ and thus $y \not\in A^{*V}$ as desired.

On the other hand, suppose $(x, U) \not\le_\alpha (y, V)$. Then by the assumption we can find a nonempty open $U' \subseteq U$ such that, fixing an enumeration $V'_n$ of nonempty basic open subsets of $V$, we can find $\BPi^0_\beta$ sets $A_n$ such that $x \in A_n^{*U'}$ but $y\in (X \setminus A_n)^{\Delta V_n'}$. Define $A = \bigcup_n (X \setminus A_n)$, which is $\BSigma^0_\beta$ and thus $\BPi^0_\alpha$. Observe that $y \in A^{*V}$ but $x \not\in A^{*U}$ by basic properties of the Vaught transforms.
\end{proof}

A theorem similar to the following appears in \cite{HKL_1998}, though it is specific to actions of non-Archimedean Polish groups. We prove a more general result, for arbitrary Polish group actions. For any $x \in X$ and open $V \subseteq G$, say that $G \cdot x$ is \textbf{densely $\BPi^0_\alpha$ in $V \cdot x$} iff there is a $\BPi^0_\alpha$ set $A \subseteq G \cdot x$ such that $x \in A^{*V}$.

\begin{proposition}\label{pr:isom_inv}
For any open $V \subseteq G$ and any $x, y \in X$ the following are all equivalent, excluding statement (1) in the special case that $\alpha = 0$:
\begin{enumerate}
    \item For every $\BPi^0_\alpha$ set $C$, $x \in C^{\Delta V}$ iff $y \in C^{\Delta V}$;
    \item For every $\BSigma^0_{\alpha+1}$ set $C$, $x\in C^{\Delta V}$ iff $y \in C^{\Delta V}$; and
    \item For every $\BPi^0_{\alpha+1}$ set $C$, $x \in C^{* V}$ iff $y \in C^{* V}$.
\end{enumerate}

If moreover both $G \cdot x$ and $G \cdot y$ are densely $\BPi^0_{\alpha+2}$ in $V \cdot x$ and $V \cdot y$ respectively, then any of these statements imply that $y \in G \cdot x$.
\end{proposition}

\begin{proof}
To see that (3) implies (2), suppose $C$ is $\BSigma^0_{\alpha+1}$ and $x \in C^{\Delta V}$. Then we have $x \not\in (X \setminus C^{\Delta V})$. By basic properties of the Vaught transform (see \cite[Lemma 5.1.7]{Becker_Kechris_1996}), we compute that $X \setminus C^{\Delta V} = (X \setminus C)^{* V}$ is $\BPi^0_{\alpha+1}$. Thus we have $y\not\in (X \setminus C)^{*V}$ and therefore $y \in C^{\Delta V}$. The fact that (2) implies (3) is the same argument.

Statement (2) trivially implies (1). To see that (1) implies (2), suppose again that $C$ is  $\BSigma^0_{\alpha+1}$ and $x \in C^{\Delta V}$. 
Write $C$ as a countable union $C = \bigcup_k C_k$ where each $C_k$ is $\BPi^0_\alpha$. 
By the Baire category theorem we have that $x\in C_k^{\Delta V}$ for some $k$. Thus we have that $y \in C_k^{\Delta V}$ and therefore $y\in C^{\Delta V}$.

We now show that under the additional assumption, statement (2) implies that $y \in G \cdot x$. Let $A \subseteq G \cdot x$ and $B \subseteq G \cdot y$ be $\BPi^0_{\alpha+2}$ sets such that $x \in A^{*V}$ and $y \in B^{*V}$.

By the assumption on the Borel complexities, we may write 
\[A = \bigcap_{k \in \omega} \bigcup_{\ell \in \omega} A_{k, \ell}\]
and 
\[B = \bigcap_{k \in \omega} \bigcup_{\ell \in \omega} B_{k, \ell}\]
where each $A_{k, \ell}$ and $B_{k \ell}$ are $\BPi^0_\alpha$. By basic properties of the Vaught transforms, fixing some countable basis $\mathcal{B}$ for the topology on $\Gspace$, we may write:
\[A = A^{* \Gspace} = \bigcap_{k \in \omega} \left[\bigcup_{\ell \in \omega} A_{k, \ell}\right]^{* \Gspace} = \bigcap_{k \in \omega} \bigcap_{W} \left[\bigcup_{\ell \in \omega} A_{k, \ell}\right]^{\Delta W} = \bigcap_{k \in \omega} \bigcap_{W} \bigcup_{\ell \in \omega} A_{k, \ell}^{\Delta W} \]
where $W$ ranges over nonempty basic open subsets of $G$.
Similarly, we may write:
\[G \cdot y = \bigcap_{k \in \omega} \bigcap_{W} \bigcup_{\ell \in \omega} B_{k, \ell}^{\Delta W}. \]

We seek a new topology $\sigma$ on $\Xspace$, in which $(\Xspace, \sigma)$ is still a Polish $\Gspace$-space, and $A_{k, \ell}^{\Delta W}$ and $B_{k, \ell}^{\Delta W}$ are open for every $k, \ell \in \omega$ and basic open $W$. Under such a topology, $A$ and $B$ will be $\BPi^0_2$ by the previous paragraph. By Proposition \ref{pr:hjorth_topology_change}, such a topology exists, in which every basic open is $\BSigma^0_{\alpha+1}$ with respect to the original topology. By the assumption (3), $\overline{V\cdot x}^\sigma = \overline{V \cdot y}^\sigma$ and $G \cdot x$ and $G \cdot y$ are comeager in $\overline{V\cdot x}^\sigma$ and $\overline{V \cdot y}^\sigma$ respectively. Thus they must have nonempty intersection, in particular $y \in G \cdot x$.
\end{proof}

\section{Actions of non-Archimedean TSI Polish groups}

Given a sequence $(\Delta_i)_{i \in \omega}$ of countable groups, as in \cite{Solecki_1995}, define $\Delta^n$ to be the product $\prod_{i \le n} \Delta_i$. 
A {\bf group tree} is a subset $T \subseteq \bigcup_{n \in \omega} \Delta^n$ for some sequence $(\Delta_i)_{i \in \omega}$ such that $T$ is closed under initial segments (i.e. $T$ is a tree), and for every $n$, $T^n := T \cap \Delta^n$ is a subgroup of $\Delta^n$.

For any $\oa \in T$, let $\ell(\oa)$ be the length of $\oa$. Given two sequences $\oa$ and $\ob$ of the same length $n$, let $\oa \cdot \ob$ be the group product of $\oa$ and $\ob$ as elements of $T^n$. We will also let $\oa^{-1}$ refer to the inverse of $\oa$ as an element of $T^n$, and $1^n$ refer to the identity element of $T^n$. Let $T_\oa$ be the set of $\ob$ such that $\oa \smf \ob \in T$, where $\oa \smf \ob$ is the concatination of $\oa$ and $\ob$.
Similarly, let $T_\oa^k$ be the set of $\ob \in T_\oa$ of length exactly $k$.

If $T \subseteq \bigcup_{n \in \omega} \Delta^n$ is a group tree, define $\Gspace_T$ to be the closed (and hence Polish) subgroup of $\prod_{n \in \omega} \Delta_n$ consisting of the branches through $T$.
For every $\oa \in T$, let
\[\mcV^T_\oa := \{g \in \Gspace_T \mid \oa \sqsubseteq g\}.\]
Then $\{\mcV^T_{1^n} \mid n \in \omega\}$ forms a countable local basis for the identity of $\Gspace_T$ of open normal subgroups, and $\{\mcV^T_\oa \mid \oa \in T\}$ is a countable basis. We will assume, without loss of generality, that $T$ is \emph{pruned}, i.e. every sequence in $T$ can be properly extended in $T$, so that each $\mcV^T_\oa$ is nonempty when $a \in T$. We use $A^{*a}$ as shorthand notation for $A^{*V}$ where $V = V^T_a$.

Any non-Archimedean TSI Polish group $G$ is isomorphic to a closed subgroup of $\prod_{i = 0}^\infty \Delta_i$ for some choice of $(\Delta_i)_{i \in \omega}$ (see e.g. \cite[Lemma 12.1]{Hjorth_Kechris_1997}). This tells us that we can write $G$ as $G_T$ for some group tree $T$.

\subsection{Hjorth analysis for non-Archimedean TSI Polish groups}

Assume that $G$ is a non-Archimedean TSI Polish group, so fix a group tree $T$ such that $G_T$ is homeomorphic to $G$. For the rest of this section, fix a Polish $G_T$-space $X$.

We will now define relations $\sim^{T, X}_\alpha$ on $X$ for $\alpha < \omega_1$. 
We will define them recursively on $\alpha$, and the recursive definitions will actually require us to define them on the slightly larger space $X \times T$. 
We may loosely view the relation $\sim^{T, X}_\alpha$ as asserting $\alpha$-back-and-forth equivalence in the sense of the Scott analysis (see \cite[Chapter 12]{Gao_2009}).
We will also need to define ``relativized" relations $\sim^{T, X}_{\alpha, n}$ for $n \in \omega$, where $\sim^{T, X}_\alpha$ is identified with $\sim^{T, X}_{\alpha, 0}$.

\begin{definition}
By transfinite recursion, we define relations $\sim^{T, X}_{\alpha, n}$ on $X \times T$ for ordinals $\alpha$ and $n \in \omega$. Say that
\[(x, \oa) \sim^{T, X}_0 (y, \ob) \quad \text{iff} \quad \overline{\mcV^T_\oa \cdot x} = \overline{\mcV^T_\ob \cdot y}\; \text{and} \; \ell(\oa) = \ell(\ob). \]
Assuming $\sim^{T, X}_\alpha$ is defined, we say
\[(x, \oa) \sim^{T, X}_{\alpha, n} (y, \ob) \quad \text{iff} \quad \exists \oc \in T^n_\oa, \exists \od \in T^n_\ob, (x, \oa \smf \oc) \sim^{T, X}_{\alpha} (y, \ob \smf \od).\]
Finally, assuming $\sim^{T, X}_{\beta, m}$ is defined for every $\beta < \alpha$ and $m \in \omega$, then
\[(x, \oa) \sim^{T, X}_{\alpha} (y, \ob) \quad \text{iff} \quad \forall \beta < \alpha, \forall m \in \omega, \; (x, \oa) \sim^{T, X}_{\beta, m} (y, \ob). \]
\end{definition}

We collect a few useful facts about these relations, which are easily verified.
\begin{lemma}\label{sim_facts}
For any ordinal $\alpha$ and $n \in \omega$:
\begin{enumerate}
    \item If $(x, \oa) \sim^{T, X}_{\alpha, n} (y, \ob)$ then $(x, \oa \cdot \oc) \sim^{T, X}_{\alpha, n} (y, \ob \cdot \oc)$ whenever $\ell(c) = \ell(a) = \ell(b)$, and thus in particular the relations $\sim^{T, X}_{\alpha, n}$ are equivalence relations;
    \item If $(x, \oa) \sim^{T, X}_{\alpha, n} (y, \ob)$ then $(x, \oa) \sim^{T, X}_{\alpha, m} (y, \ob)$ for any $m < n$; and
    \item If $(x, \oa) \sim^{T, X}_{\alpha, n} (y, \ob)$ then $(x, \oa) \sim^{T, X}_{\beta, m} (y, \ob)$ for any $\beta < \alpha$ and any $m \in \omega$.
\end{enumerate}
\end{lemma}

Working towards a lower bound, the next lemma shows the relationship between the relations $\sim^{T, X}_\alpha$ and the Hjorth relations.

\begin{lemma}\label{lm:hjorth_isom}
Fix $(x, \oa), (y, \ob) \in X \times T$ with $\ell(\oa) = m$, $\ell(\ob) = n$, and $m \ge n$ such that $(x, \oa\upharpoonright n) \sim^{T, X}_{\alpha} (y, \ob)$. Then we have that $(x, \mcV^T_{\oa}) \le_{\alpha'} (y, \mcV^T_{b})$ where $\alpha' = \alpha+1$ for $\alpha < \omega$ and $\alpha' = \alpha$ for $\alpha \ge \omega$.
\end{lemma}

\begin{proof}



The base case is trivial, so we proceed to the induction step. Suppose for some $1 \le n \in \omega$ the claim is true below $n$. Assume first that 
\[(x, \oa \upharpoonright n) \sim^{T, X}_n (y, \ob).\]
We will show that 
\[(x, \mcV^T_{a}) \le_{n+1} (y, \mcV^T_{b}).\]
Fix an arbitrary nonempty open subset $\mcW$ of $\mcV^T_{a}$. 
Find some $c \in T_a$ such that $\mcW \supseteq \mcV^T_{\oa \smf \oc}$. 
By assumption we can find some $\od \in T_\ob$ such that 
\[(x, \oa \smf \oc) \sim^{T,X}_{n-1} (y, \ob \smf \od),\]
and thus 
\[(y, \mcV^T_{\ob \smf \od}) \le_\beta (x, \mcV^T_{\oa \smf \oc})\]
by the induction hypothesis, and so
\[ (y, V^T_{b \smf d}) \le_\beta (x, W).\]

The rest is easy.

\end{proof}


The next two propositions collect what we have proved about the relations $\sim^{T, X}_\alpha$ and $\sim^{T, X}_{\alpha, n}$.

\begin{proposition}\label{pr:isom_sim}
For any $(x, a), (y, b) \in X \times T^k$, $k < \omega$, $\lambda < \omega_1$ a limit ordinal and $n \in \omega$,
\begin{enumerate}
    \item If $(x, a) \sim^{T, X}_n (y, \ob)$, then for any $\BPi^0_{n+1}$ set $C$, $x \in C^{* a} \; \text{iff} \; y \in C^{* b}$. If in addition $G \cdot x$ and $G \cdot y$ are densely $\BPi^0_{n+2}$ in $\overline{V_{1^k} x}$ and $\overline{V_{1^k} y}$ respectively, then $y \in G_T x$; and
    \item If $(x, a) \sim^{T, X}_{\lambda+n} (y, \ob)$, then for any $\BPi^0_{\lambda+n}$ set $C$, $x \in C^{* a} \; \text{iff} \; y \in C^{* b}$. If in addition $G \cdot x$ and $G \cdot y$ are densely $\BPi^0_{\lambda+n+1}$ in $\overline{V_{1^k} x}$ and $\overline{V_{1^k} y}$ respectively, then $y \in G_T x$.
\end{enumerate}

In particular, if $E^G_X$ is $\BPi^0_{n+2}$ and $x \sim^{T, X}_n y$ then $y \in G_T x$. If $E^G_X$ is $\BPi^0_{\lambda+n+1}$ and $x \sim^{T, X}_{\lambda+n} y$ then $y \in G_T x$.
\end{proposition}

\begin{proof}
This is just Proposition \ref{pr:isom_inv}, Proposition \ref{pr:hjorth_inv}, and Lemma \ref{lm:hjorth_isom}.
\end{proof}

In the next proposition, define
\[A^{T, X}_{n, k} := \{x \in X \mid G\cdot x \; \text{is densely} \; \BPi^0_{n+2} \; \text{in} \; V_{1^k}x \}\]
for $n \in \omega$ and $k \in \omega$, and define
\[A^{T, X}_{\lambda+n, k} := \{x \in X \mid G\cdot x \; \text{is densely} \; \BPi^0_{\lambda+n+1} \; \text{in} \; V_{1^k}x \}\]
for limit $\lambda$, $n \in \omega$ and $k \in \omega$.

\begin{proposition}\label{pr:isom_sim2}
For any $x, y \in X$, $k < \omega$, $\lambda < \omega_1$ a limit ordinal and $n \in \omega$.
\begin{enumerate}
    \item If $x, y \in A_{n, k}^{T, X}$ and $x \sim^{T, X}_{n, k} y$ then $y \in G_T \cdot x$;
    \item If $x, y \in A_{\lambda+n, k}^{T, X}$ and $x \sim^{T, X}_{\lambda+n, k} y$ then $y \in G_T \cdot x$.
\end{enumerate}

\end{proposition}

\begin{proof}
We only prove statement (1), as statement (2) is essentially the same.

Fix $\BPi^0_{n+2}$ sets $A \subseteq G_T \cdot x$ and $B \subseteq G_T \cdot x$ such that $x \in A^{*1^k}$ and $y \in B^{*1^k}$, and fix $a, b \in T^k$ such that $(x, a) \sim^{T, X}_n (y, b)$. 
For any $g \in V_a$ and $h \in V_b$, we have $(g^{-1}x, 1^k) \sim^{T, X}_n (h^{-1} y, 1^k)$.
By Proposition \ref{pr:isom_sim}, $y \in G_T x$.

\end{proof}

\subsection{A universal $G_T$-space}

Let $\mcC$ represent Cantor space. Given $c$ a partial functon from $T$ to $2^{<\omega}$, let $\mathcal{N}_c$ be the set of $x \in \mcC^T$ such that for every $\oa$ in the domain of $f$, $c(a) \sqsubseteq x(\oa)$. The family of such $\mathcal{N}_c$ forms a countable basis for a Polish topology on $\mcC^T$.  

We will be concerned with the Polish space $\mcC^T$ of functions from the countable discrete space $T$ to $\mcC$ with the pointwise convergence topology.
There is a natural action of $G_T$ on $\mcC^T$ by shifts, i.e.
\[ g \cdot x = \bfx^g, \; \text{where} \; \bfx^g(\ob) = x((g \upharpoonright \ell(\ob))^{-1}\ob). \]
which is continuous with respect to this topology.

\begin{theorem}\label{th:universal_map}
Let $\Xspace$ be a Polish $\Gspace_T$-space. Then there is a Borel injective map $f : \Xspace \rightarrow \mcC^T$ which is equivariant, i.e.
\[ \forall g \in \Gspace_T, \forall x \in \Xspace, g \cdot f(x) = f(g \cdot x). \]
\end{theorem}

\begin{proof}
Fix a basis $\{U_n \mid n \in \omega\}$ for the topology on $X$. Let $f : \Xspace \rightarrow \mcC^T$ be the function defined by 
\[f(x) = \bfx \text{ where } \bfx(\oa)(n) = 1 \text{ iff }  (V_{a^{-1}}^T \cdot x) \cap U_n \neq \emptyset.\]
We claim that $f$ works. Injectivity follows from the fact that $\Xspace$ is Hausdorff and the action is continuous. To see that $f$ is Borel, fix an arbitrary finite partial function $c : T \rightarrow 2^{<\omega}$ and observe
\begin{equation}\label{eq:1}
f^{-1}[\mcU_c] = \bigcap \big[\{V_{a^{-1}} \cdot U \mid c(\oa)(n) = 1\} \cup \{\Xspace \setminus (V_{a^{-1}} \cdot U) \mid c(\oa)(n) = 0\}\big],
\end{equation}
which is Borel. To check equivariance, fix arbitrary $g \in \Gspace_T$, $x \in \Xspace$, $n \in \omega$, and $a \in T$ of length $k$, and observe that
\begin{align*}
    &[g \cdot f(x)](\oa)(n) = 1 \\
    \iff &f(x)((g \upharpoonright k)^{-1}a)(n) = 1 \\
    \iff &(V^T_{a^{-1}(g \upharpoonright k)} \cdot x) \cap U_n \neq \emptyset \\
    \iff &(V^T_{a^{-1}} \cdot gx) \cap U_n \neq \emptyset \\
    \iff &f(gx)(a)(n) = 1
\end{align*}
as desired.
\end{proof}

We will define relations $\sim^T_{\alpha, n}$ on the Polish $G_T$-space $\mcC^T \times T$. While these relations won't literally coincide on $\mcC^T \times T$ even when $X$ is taken to be $\mcC^T$, they are equally descriptive with respect to Borel reducibility, and the relations we define now are purely algebraic and easier to work with in some future arguments. The only difference in the recursive definition is in the base case, though of course this difference propagates and yields different relations at every level.

\begin{definition}
By transfinite recursion, we define relations $\sim^T_{\alpha, n}$ on $\mcC^T \times T$ for ordinals $\alpha$ and $n \in \omega$. Say that
\[(x, \oa) \sim^{T, X}_0 (y, \ob) \quad \text{iff} \quad \ell(\oa) = \ell(\ob)\; \text{and} \; x(a) = y(a). \]
Assuming $\sim^T_\alpha$ is defined, we say
\[(x, \oa) \sim^T_{\alpha, n} (y, \ob) \quad \text{iff} \quad \exists \oc \in T^n_\oa, \exists \od \in T^n_\ob, (x, \oa \smf \oc) \sim^T_{\alpha} (y, \ob \smf \od).\]
Finally, assuming $\sim^T_{\beta, m}$ is defined for every $\beta < \alpha$ and $m \in \omega$, then
\[(x, \oa) \sim^T_{\alpha} (y, \ob) \quad \text{iff} \quad \forall \beta < \alpha, \forall m \in \omega, \; (x, \oa) \sim^T_{\beta, m} (y, \ob). \]
\end{definition}

\begin{proposition}\label{pr:sim_reduce}
Let $X$ be a Polish $G_T$-space. Then for any $\alpha$ and $n \in \omega$,
\[ \sim^{T,X}_\alpha \; \le_B \; \sim^T_\alpha\]
and
\[ \sim^{T,X}_{\alpha, n} \; \le_B \; \sim^T_{\alpha, n}. \]
\end{proposition}

\begin{proof}
Let $f : X \rightarrow \mcC^T$ be the equivariant injection from Theorem \ref{th:universal_map}. We see that this serves as the reduction for every $\alpha$. It suffices to check that by the definition of $f$ we have $f(x)(a) = f(y)(b)$ iff $\overline{V_{a^{-1}}^T x} = \overline{V_{b^{-1}}^T y}$. The rest follows by an easy induction on $\alpha$ using the equivariance of $f$.
\end{proof}

Given a sequence $(R_n)_n$ of relations living on Polish spaces $X_n$ let $\bigsqcup_n R_n$, the \emph{disjoint union} of the sequence $(R_n)_n$, be the relation living on the disjoint union $\bigsqcup_n X_n$ that relates $x$ and $y$ if for some $n$, $x, y \in X_n$ and $x R_n y$. 

We make a few easy observations about potential Borel complexity before continuing, which we will use repeatedly in Section \ref{sec:applications}. Assuming that each $R_n$ is potentially $\BPi^0_\alpha$, then so are the product $\prod_n R_n$ and the disjoint union $\bigsqcup_n R_n$. The converse holds as well, as trivially $R_m \le_B \bigsqcup_n E_n$, and assuming the $R_n$ are reflexive, each $R_m \le_B \prod_n R_n$. 

On the other hand, assuming the relations $R_n$ all live on the same Polish space $X = X_n$, we have $\bigcup_n R_n$ is potentially $\BSigma^0_\alpha$ assuming each $R_n$ is $\BSigma^0_\alpha$. We also have that $\bigcap_n R_n$ is potentially $\BPi^0_\alpha$ assuming each $R_n$ is $\BPi^0_\alpha$.

For the next proposition we define
\[ \sim^T_{\alpha, <\omega} \; := \bigsqcup_n \sim^T_{\alpha, n} \]
and
\[ \sim^T_{< \alpha} \; := \bigsqcup_{\beta < \alpha} \sim^T_{\beta}. \]

\begin{proposition}\label{prop:upper_bounds}
For any Polish $G$-space with $G = G_T$, the following all hold for any $n \in \omega$ and limit $\lambda$:
\begin{enumerate}
    \item if $E^G_X$ is potentially $\BPi^0_{n+2}$ then $E^G_X \le_B \; \sim^T_n$, and if $E^G_X$ is potentially $\BPi^0_{\lambda+n+1}$ then $E^G_X \le_B \; \sim^T_{\lambda+n}$;
    \item if $E^G_X$ is potentially $\BSigma^0_{n+3}$ then $E^G_X \le_B \;\sim^T_{n, <\omega}$ and if $E^G_X$ is potentially $\BSigma^0_{\lambda+n+2}$ then $E^G_X \le_B \; \sim^T_{\lambda + n, <\omega}$; and
    \item if $E^G_X$ is potentially $\bigoplus_{\alpha < \lambda} \BPi^0_\alpha$ then $E^G_X \le_B \;\sim^T_{<\lambda}$.
\end{enumerate}
\end{proposition}

\begin{proof}
Of course, we may assume in any case that $E^G_X$ is exactly the desired complexity, by the change of topology argument in Theorem \ref{th:top_change}.

First, we prove part (1). Suppose $E^{G_T}_X$ is $\BPi^0_{n+2}$. Then the relations $\sim^{T, X}_n$ and $E^{G}_X$ coincide by Proposition \ref{pr:isom_sim}. Thus $E^G_X$ is Borel-reducible to $\sim^T_n$ by Proposition \ref{pr:sim_reduce}. The infinite case is essentially the same.

We now proceed to part (2). Suppose $E^{G}_X$ is $\BSigma^0_{n+3}$. Write $E^G_X = \bigcup_\ell A_\ell$ where each $A_\ell$ is $\BPi^0_{n+2}$. 
For every $\ell$ and $k$, let
\[A_{\ell, k} := \{x \in X \mid \exists a, b \in T^k, \; x \in A_{\ell}^{*a, b}\} \]
which is Borel, and also invariant by the continuity of the action.
Observe that by continuity of the group action, if $x \in A_{\ell, k}$ then $G \cdot x$ is densely $\BPi^0_{n+2}$ in $V^T_{1^k} \cdot x$.
For any $x$, we have $(x, x) \in (E^G_X)^{\Delta G \times G}$ where we consider the natural product action of $G \times G$ on $X \times X$, and thus $(x, x) \in A_\ell^{*a, b}$ for some $k$ and $a, b\in T^k$.
We just observed that the $A_{\ell, k}$ cover $B$, and moreover by Proposition \ref{pr:isom_sim2}, $\sim^{T, X}_{n, k}$ and $E^G_X$ coincide on $A_{\ell, k}$. Thus $E^G_X \le_B \;\mathrel{\sim^{T}_{n, <\omega}}$ by Proposition \ref{pr:sim_reduce}. The infinite case is essentially the same.

For part (3), fix a partition $X = \bigsqcup_{\alpha < \lambda} X_\alpha$ of $G$-invariant Borel sets such that for every $x \in X_\alpha$, $G \cdot x$ is $\BPi^0_\alpha$. On $X_\alpha$, the relations $\sim^{T, X}_\alpha$ and $E^G_X$ coincide, thus $E^G_X \le_B \;\sim^T_{<\lambda}$ by Proposition \ref{pr:sim_reduce}.
\end{proof}

We establish some upper bouds for potential Borel complexity. By Proposition \ref{pr:hkl} and the following proposition, the results of the prevous proposition are tight.
\begin{proposition}\label{prop:friedman_stanley}For every $\alpha < \omega_1$, $\sim^T_\alpha$ is Borel reducible to $=^{+(\alpha)}$.
\end{proposition}

\begin{proof}
Let $\mcC$ represent Cantor space. We will prove this by describing how, for each ordinal $\alpha$, to associate each $(x, \oa) \in X \times T$ with a set $A^\alpha_{x, \oa} \in \mathcal{P}^\alpha(\mcC)$ such that $(x, \oa) \sim^T_\alpha (y, \ob)$ iff $A^\alpha_{x, \oa} = A^\alpha_{y, \ob}$.

For any $(x, \oa) \in X \times T$, the $\sim^T_0$-class of $(x, \oa)$ is determined entirely by the real number $x(\oa)$.
This defines a Borel reduction of $\sim^T_0$ to equality on the reals via the map $\pi : X \times T \rightarrow \mcC$ defined by
\[ (x, \oa) \mapsto c_{x, \oa} \quad \text{where} \quad c_{x, \oa} = x(\oa).\]

Now suppose for some $\alpha$ that the reduction $\pi_\beta$ has been defined for $\beta$ for every $\beta < \alpha$. For every $(x, \oa) \in X \times T$ and $\beta < \alpha$, let $A^\beta_{x, \oa}$ be the element of $\mathcal{P}^\beta(\mcC)$ that serves as an invariant for its $\sim^T_\beta$-class. Then for any $(x, \oa)$, the invariant for its $\sim^T_{\alpha+1}$ class is simply the set
\[ A^{\alpha}_{x, \oa} :=  \{A^\alpha_{(x, \oa \smf \ob)} \mid \ob \in T_\oa \}. \]
We can then define $\pi_{\alpha+1}$ by
\[ (x, a) \mapsto \langle \pi_\alpha(x, a \smf b_{a, n}) \mid n < \omega \rangle \]
where $b_{a, n}, n \in \omega$ is a fixed enumeration of $T_a$ for each $a \in T$.
The fact that this is a reduction is clear from the definition of $\sim^T_{\alpha+1}$ and Lemma \ref{sim_facts}.(1).

For limit $\alpha$, merely define $\pi_\alpha(x, \oa) = \langle \pi_{\beta}(x, \oa) \mid \beta < \alpha \rangle$.
\end{proof}

\subsection{A family of actions of non-Archimedean TSI groups}

For this section, fix a group tree $T$. We show that the relations $\sim^T_\alpha$ and $\sim^T_{\alpha, <\omega}$ are classifiable by $(G_T)^\omega$. It's not clear if in general we can get classifiability by $G_T$. Arguments in \cite{Hjorth_2010} suggest that it likely not possible. However of course we do get classifiability by $G_T$ assuming that $(G_T)^\omega \isom G_T$, such as $F_2^\omega$ or $\mathbb{Z}^\omega$. This would also be useful if one would like to study actions of products of solvable groups, or amenable groups, and so on.

Recall that $T^m$ is the countable discrete group of sequences in $T$ of length $m$. For each $k < m$, we will define $T^m_k$ to be the normal subgroup of $T^m$ defined by:
\[T^m_k := \{\oa \in T^m \mid \oa \upharpoonright k = 1^k\}.\]

For a countable group $\Gamma$, the $\Gamma$-shift of a Polish $\Gspace$-space $(\Xspace, a)$, is the Polish $(\Gamma \times \Gspace)$-space $\Xspace^\Gamma$ with the action defined by:
\[((a, g)\cdot x)(b) = g \cdot x(a^{-1}b). \]

If for each $n$, there is a Polish group $\Gspace_n$ and a Polish $\Gspace_n$-space $(\Xspace_n, a_n)$, write $\prod_n (\Xspace_n, a_n)$ to refer to be the Polish $\prod_n \Gspace_n$-space living on $\prod_n \Xspace_n$ with the induced product action $(g_n)_n \cdot (x_n)_n = (g_n \cdot_{a_n} x_n)_n$. Write $\bigsqcup_n (\Xspace_n, a_n)$ to be the Polish $\prod_n \Gspace_n$-space living on $\bigsqcup_n \Xspace_n$ with the action induced by restrictions, i.e. $(g_n)_n \cdot x = g_k \cdot_{a_k} x$ for the unique $k$ such that $x \in X_k$. Of course we may also take products and unions over arbitrary countable ordinals, i.e. if $X_\beta$ is a Polish $G_\beta$-space for every $\beta < \alpha$, then $\prod_{\beta < \alpha} X_\beta$ and $\bigsqcup_{\beta < \alpha} X_\beta$ are Polish $\prod_{\beta < \alpha} G_\beta$-spaces.

For each ordinal $\alpha \in \omega_1$ and $n, m \in \omega$, we will define a Polish group $G^m_{\alpha, n}$ and a Polish $G^m_{\alpha, n}$-space $(X^m_{\alpha, n}, a^m_{\alpha, n})$.

\begin{definition} For $n, m \in \omega$ and $\alpha \ge 1$, we define the following:
\begin{enumerate}
    \item Let $(X^m_{0, n}, a^m_{0, n})$ be the shift action of $T^{m+n}_m$ on $\mcC^{T^{m+n}_m}$, where $\mcC$ is Cantor space;
    \item Assuming $(X^{m'}_{\beta, n'}, a^{m'}_{\beta, n'})$ is defined for every $m', n' \in \omega$ and $\beta < \alpha$, we define $(X^m_{\alpha, n}, a^m_{\alpha, n})$ to be the $T^{m+n}_m$-shift of
    \[\prod_{i \ge n} \prod_{\beta < \alpha} (\Xspace^{m+n}_{\beta, i}, a^{m+n}_{\beta, i}).\]
    
\end{enumerate}
\end{definition}

By an easy induction we see that each $(X^m_{\alpha, n}, a^m_{\alpha, n})$ is a Polish $G^m_{\alpha, n}$-space where $G^m_{0, n}$ is $T^{m+n}_m$ and for $\alpha > 0$, $G^m_{\alpha, n}$ is
\[T^{m+n}_m \times \prod_{i \ge n} \prod_{\beta < \alpha} G^{m+n}_{\beta, i}\Big.\]

For each $n, m \in \omega$ and ordinal $\alpha$, let $E^m_{\alpha, n}$ be the orbit equivalence relation of the $G^m_{\alpha, n}$-space $(\Xspace^m_{\alpha, n}, a^m_{\alpha, n})$. Let $E_{\alpha, <\omega}$ be the orbit equivalence relation of $\bigsqcup_n (\Xspace^0_{\alpha, n}, a^0_{\alpha, n})$. For limit $\lambda$, let $E_{<\lambda}$ be the orbit equivalence relation of $\bigsqcup_{\beta < \alpha} (\Xspace^0_{\beta, 0}, a^0_{\beta, 0})$.

\begin{lemma}\label{lm:sim_reduce}
For every Polish $G_T$-space $X$, each ordinal $\alpha$ and $n, m \in \omega$, $\sim^T_{\alpha, n} \upharpoonright (\mcC^T \times T^m)$ is Borel-reducible to $E^m_{\alpha, n}$.
\end{lemma}

\begin{proof}
First we start with $\alpha = 0$. For $n, m \in \omega$, the reduction $f^m_{0, n} : X \times T^m \rightarrow X^m_{0, n}$ is defined by
\[ (x, a) \mapsto c^{x, a} \quad \text{where} \quad c^{x, \oa}(1^m \smf b) = x(a \smf b).\]

Now, given $\alpha > 0$, $n, m \in \omega$, the reduction $f^\alpha_{n, m} : \mcC^T \times T^m \rightarrow X^m_{\alpha, n}$ is defined by
\[ (x, a) \mapsto c^{x, a} \quad \text{where} \quad c^{x, \oa}(1^m \smf b)(i)(\beta) = f^\beta_{i, n+m}(a\smf b).\]
It is routine to check that these work.
\end{proof}

\begin{proposition}\label{prop:tsi_actions}
The relations $\sim^T_{\alpha}$,  $\sim^T_{\alpha, <\omega}$, and $\sim^T_{<\lambda}$ are classifiable by $(G_T)^\omega$ for every ordinal $\alpha$, limit ordinal $\lambda$ and $n \in \omega$.
\end{proposition}

\begin{proof}
We say that a Polish group $G$ involves a Polish group $H$ iff there is a continuous surjective homomorphism from a closed subgroup of $G$ onto $H$.

Observe that $G^T$ involves $T^n_k$ for every $0 \le k \le n$. By induction, $(G_T)^\omega$ involves every $G^m_{\alpha, n}$. By \cite[Theorem 2.3.5]{Becker_Kechris_1996}, each $E^m_{\alpha, n}$ is Borel bi-reducible with an orbit equivalence relation induced by a continuous action of $(G_T)^\omega$.
\end{proof}

\subsection{Universal relations for each potential Borel complexity class}

We now collect the previous results all into one statement:

\begin{proposition}\label{pr:fine_universal}
For any group tree $T$ such that $(G_T)^\omega \isom G_T$ the following all hold:
\begin{enumerate}
    \item For each $n \in \omega$ and limit $\lambda < \omega_1$,
    \begin{enumerate}
        \item $\sim^T_{n}$ is universal for potentially $\BPi^0_{n+2}$ equivalence relations classifiable by $G_T^\omega$;
        \item $\sim^T_{\lambda+n}$ is universal for potentially $\BPi^0_{\lambda+n+1}$ equivalence relations classifiable by $G_T^\omega$;
    \end{enumerate}
    
    \item For each $n \in \omega$ and limit $\lambda < \omega_1$,
    \begin{enumerate}
        \item $\sim^T_{n, <\omega}$ is universal for potentially $\BSigma^0_{n+3}$ equivalence relations classifiable by $G_T^\omega$;
        \item $\sim^T_{\lambda + n, <\omega}$ is universal for potentially $\BSigma^0_{\lambda+n+2}$ equivalence relations classifiable by $G_T^\omega$;
    \end{enumerate}
    \item For each limit $\lambda$, $\sim^T_{<\lambda}$ is universal for potentially $\bigoplus_{\alpha < \lambda} \BPi^0_\alpha$ equivalence relations classifiable by $G_T^\omega$;
\end{enumerate}
\end{proposition}

By Proposition \ref{pr:hkl}, we may consult (*) and get equivalent but tighter statements of universality. For example, we have that $\sim^T_\lambda$ is universal for potentially $\BPi^0_{\lambda}$ equivalence relations classifiable by $G_T^\omega$, for $n \ge 1$ we have that $\sim^T_{n, <\omega}$ is universal for potentially $D(\BPi^0_{n+2})$ equivalence relations classifiable by $G_T^\omega$, and so on.

\begin{proof}
We showed that these relations are are classifiable by $G_T^\omega$ in Proposition \ref{prop:tsi_actions}. In Proposition \ref{prop:friedman_stanley}, we showed that $\sim^T_\alpha$ is Borel-reducible to $=^{+\alpha}$ for every $\alpha$. By the calculations from \cite{HKL_1998}[Theorem 5.1] of the Friedman-Stanley jumps of equality, we can compute the complexity upper-bounds. And then an application of Corollary \ref{cor:upper_bound} tells us that the above relations are Borel-reducible to actions of $G_T^\omega$ with the desired complexity.

The rest follows by Proposition \ref{prop:upper_bounds}.
\end{proof}

\section{Applications}\label{sec:applications}

We now move on to proving the theorems and corollaries promised in the introduction.

\subsection{An obstruction to classifiability}
We prove the theorems relating to  obstructions to classification by non-Archimedean TSI actions. 

\begin{proposition}\label{prop:generic_ergodicity}
Suppose $E$ is an equivalence relation which is generically $E_\infty$-ergodic. Then $E$ is generically $\sim^T_\alpha$-ergodic, for any countable ordinal $\alpha$ and group tree $T$.
\end{proposition}

\begin{proof}

We first need to prove the following technical claim.
\begin{claim}\label{claim:squiggle_quotient}
For any group tree $T$, countable ordinal $\alpha$, $z \in \mcC^T$, and $n < \omega$,
\[ \sim^T_{\alpha, n} \upharpoonright [z]_{\sim^T_\alpha} \le_B E_\infty.\]
\end{claim}

\begin{proof}
Let $L := \prod_{\beta < \alpha} \prod_{k \in \omega} 2^{T^{n+k}}$. Define a function $g : \mcC^T \rightarrow L^{T^n}$ such that for $x \in \mcC^T$ and $d \in T^n$,
\[ g(x)(d) = \ell_{x, d}\]
where for every $\beta < \alpha$, $k \in \omega$, and $e \in T^{n+k}$,
\[ \ell_{x, d}(\beta)(k)(e) = 1 \iff \exists \of \in T^{n+k}_\od, (x, \od \smf \of) \sim^T_\beta (z, e).\]
There is a natural action of $T^n$ on $L^{T^n}$ by shifts. We claim that $g$ is a Borel reduction from $\sim^T_{\alpha, n}$ restricted to $[z]_{\sim^T_\alpha}$ to that orbit equivalence relation, which is Borel-reducible to $E_\infty$ because $L$ is Polish and $T^n$ is a countable group.

Suppose that $x \sim^T_{\alpha, n} y$. Then there is some $\od \in T^n$ such that $(x, 1^n) \sim_\alpha^T (y, \od)$. We'll see that $g(x)(1^n) = g(y)(\od)$, from which it follows by Lemma \ref{sim_facts}.(1) and the definition of $g$ that $g(x)(d') = g(y)(d \cdot d')$ for every $d' \in T^n$. 

Fix some arbitrary $\beta < \alpha$ and $k \in \omega$, and find some $\oc \in T_\od^k$ such that $(x, 1^{n+k}) \sim^T_\beta (y, \od \smf \oc)$. Then for any $\ooe \in T^{n+k}$,
\begin{align*}
    &\ell_{x, n}(\beta)(k)(e) = 1 \\
    \iff &\exists \of \in T_{1^n}^k, (x, 1^n \smf \of) \sim^T_\beta (z, \ooe)\\
    \iff &\exists \of \in T_{1^n}^k, (y, \od \smf (\oc \cdot \of)) \sim^T_\beta (z, \ooe) \\
    \iff &\exists \of \in T_{\od}^k, (y, \od \smf \of) \sim^T_\beta (z, \ooe) \\
    \iff &\ell_{y, d}(\beta)(k)(e) = 1
\end{align*}
as desired, where the second implication follows from Lemma \ref{sim_facts}.(1).

On the other hand, suppose that $\ell_{x, 1^n} = \ell_{y, \od}$. We must show that for any $\beta < \alpha$ and $k \in \omega$, that $(x, 1^n) \sim^T_{\beta, k} (y, \od)$. We know $x \sim^T_{\alpha} z$, so for some $\ooe \in T^{n+k}$ we have $(x, 1^{n+k}) \sim^T_\beta (z, \ooe)$. In particular, $\ell_{x, 1^n}(\beta)(k)(e) = 1$ and thus $\ell_{y, d}(\beta)(k)(e) = 1$. Thus there exists some $\of \in T_\od^k$ such that $(y, \od \smf \of) \sim^T_\beta (z, \ooe)$ and so by transitivity we have $(x, 1^{n+k}) \sim^T_\beta (y, \od \smf \of)$ as desired.
\end{proof}

Suppose $E$ is generically $E_\infty$-ergodic. Since $\sim^T_0$ is smooth, $E$ is generically $\sim^T_0$-ergodic. Furthermore, for a given ordinal $\alpha$, if $E$ is generically $\sim^T_{\alpha, n}$-ergodic for every $n \in \omega$, it follows by the countable closure of the meager ideal that $E$ is generically $\sim^T_{\alpha+1}$-ergodic. Similarly, if $E$ is generically $\sim^T_\beta$-ergodic for every $\beta < \alpha$, with $\alpha$ countable, then it follows that $E$ is generically $\sim^T_\alpha$-ergodic.
\end{proof}

\topobstructionthm

\begin{proof}
Suppose $\Gspace$ is a non-Archimedean TSI Polish group and $\Xspace$ is a Polish $\Gspace$-space. Let $E$ be an equivalence relation on Polish space $\Yspace$ such that $E$ is generically $E_\infty$-ergodic. Fix some Baire-measurable homomorphism $f : \Yspace \rightarrow \Xspace$ from $E$ to $E^\Gspace_\Xspace$. By Theorem \ref{th:universal_map} we can assume that $\Gspace = \Gspace_T$ for some group tree $T$, and $X = \mcC^T$ with the shift action of $G_T$. 
As in Claim 5.4 of \cite{Allison_Panagio_2021}, we can find a comeager subset $C \subseteq \Yspace$ and a countable ordinal $\lambda$ such that for every $y \in C$, $G_T \cdot f(y)$ is $\BPi^0_\lambda$. 
In particular, by Proposition \ref{pr:isom_sim}, we have some countable ordinal $\lambda'$ such that the orbit equivalence relation and $\sim_{\lambda'}^T$ coincide on $f[C]$. 
By Proposition \ref{prop:generic_ergodicity}, we have a comeager subset $C' \subseteq C$ such that $f[C']$ is contained in a single $\sim^T_{\lambda'}$ class, and thus it is contained in a single $G_T$ orbit.

The other direction is immediate, because $E_\infty$ is countable and thus induced by a continuous action of a countable discrete group by Feldman-Moore, which is of course a non-Archimedean Polish group.
\end{proof} 

\ccapplication

\begin{proof}
In \cite{Clemens_Coskey}, it is shown that if $E$ is generically-ergodic and $\Delta$ is countably-infinite, then $E^{[\Delta]}$ is generically $E_\infty$-ergodic. By Theorem \ref{th:top_obstruction}, it is generically-ergodic with respect to every orbit equivalence relation induced by a non-Archimedean TSI Polish group. Every class in $E^{[\Delta]}$ is meager because $E$ is, and thus it is not classifiable by non-Archimedean TSI Polish groups.
\end{proof}

\begin{proposition}\label{prop:generic_ergodicity2}
Suppose $E$ is an equivalence relation which is generically $E_0$-ergodic. Then $E$ is generically $\sim^T_\alpha$-ergodic, for any countable ordinal $\alpha$ and abelian group tree $T$.
\end{proposition}

\begin{proof}
As in the proof of Claim \ref{claim:squiggle_quotient}, $\sim^T_{\alpha, n} \upharpoonright [z]_{\sim^T_\alpha}$ is Borel-reducible to an action of $T^n$, a countable discrete abelian group. By \cite{Gao_Jackson_2015}, it is thus Borel-reducible to $E_0$. Following the proof of Proposition \ref{prop:generic_ergodicity}, the generic ergodicity follows.
\end{proof}

Then we immediately establish following modification of Theorem \ref{th:top_obstruction}.

\topobstructionthmtwo

\subsection{Universal equivalence relations}

We now move on to the theorems relating to universality.

\tsiuniversalactionsthm

\begin{proof}
This follows from Proposition \ref{pr:fine_universal} and Theorem \ref{th:borel_reduction_inverse}.
\end{proof}

We will also see now why Corollary \ref{cor:tsi_universal_actions} follows. 

\tsiuniversalactionscor

\begin{proof}

Let $G$ be a non-Archimedean TSI Polish group, and let $T$ be a group tree on some sequence $(\Delta_i)_{i \in \omega}$ such that $G_T$ is homeomorphic to $G$. 

Since every countable group is a quotient of $F_\omega$, there is a surjective homomorphism from $(F_\omega)^\omega$ onto $\prod_i \Delta_i$. Since $G_T$ is a closed subgroup of $\prod_i \Delta_i$, we get that $(F_\omega)^\omega$ involves $G_T$.

Similarly, we get that the group $(\mathbb{Q}^{<\omega})^\omega$ involves every non-Archimedean abelian group (for this, see \cite[Corollary 2.5]{Gao_Xuan_2014}). We can then apply Mackey's theorem, \cite[Theorem 2.3.5]{Becker_Kechris_1996}, which implies that any Borel $H$-space is Borel-bireducible with some Borel $G$-space, whenever $G$ involves $H$. Thus if we apply Theorem \ref{th:tsi_universal_actions} with the group $(F_\omega)^\omega$ (resp. $(\mathbb{Q}^{<\omega})^\omega$), the resulting orbit equivalence relations are actually universal for any non-Archimedean TSI (resp. abelian) Polish group action.
\end{proof}

Next we give a more direct proof of the following result of Ding and Gao, and prove a ``higher" version of it.

\begin{corollary}
Suppose $E$ is an essentially-countable equivalence relation that is classifiable by a non-Archimedean abelian Polish group. Then $E$ is Borel-reducible to $E_0$ (i.e. is essentially hyperfinite).
\end{corollary}

\begin{proof}
In light of Corollary \ref{cor:upper_bound} and Proposition \ref{pr:fine_universal}, it is enough to see that $\sim^T_{0, <\omega}$ is essentially hyperfinite. We saw in the proof of Proposition \ref{prop:generic_ergodicity} that each $\sim^T_{0, n}$ is induced by an action of $T^n$, a countable abelian group. Thus by \cite{Gao_Jackson_2015}, $\sim^T_{0, n}$ is essentially hyperfinite. The relation $\sim^T_{0, <\omega}$ is a direct sum of essentially hyperfinite equivalence relations, and thus it too is essentially hyperfinite.
\end{proof}

The original proof relied on an application of the Seventh Dichotomy Theorem, a deep result proved in \cite{Hjorth_Kechris_1997}. Another straightforward proof of the Ding-Gao result, of a different nature, was independently given in \cite{Grebik_2020}, using the notion of \emph{$\sigma$-lacunary} actions.

Finally, we state and prove the ``higher" analogue.

\classificationcor

\begin{proof}
As in the previous proof, it is enough to observe that $\sim^T_1$ is Borel-reducible to the product of the $\sim^T_{0, n}$, each of which is countable and thus Borel-reducible to $E_\infty$. In the case that $G_T$ is non-Archimedean abelian, each $\sim^T_{0, n}$ is induced by an action of $T^n$, which is a countable abelian group, and thus by the result \cite{Gao_Jackson_2015}, is hyperfinite. And thus $E$ is Borel-reducible to $E_0^\omega$.
\end{proof}

\subsection{Borel complexity spectrum}

\tsicomplexityspectrumthm

\begin{proof}
We argue that if a complexity class $\Gamma$ in $(*)$ is not attained, then none of the ones after it are attained. Here, we say that a complexity class $\Gamma$ in $(*)$ is not attained if any orbit eqivalence relation induced by a continuous action of $\Delta^\omega$ on a Polish space that has potential complexity $\Gamma$ also has potential complexity $\Gamma'$ for some $\Gamma'$ that appears earlier in $(*)$. Let $T = \Delta^{<\omega}$. We will argue that each $\Gamma$ is attained by some $\sim^T_{\alpha}$, $\sim^T_{\alpha, <\omega}$, or $\bigoplus_{\alpha < \gamma} \sim^T_\alpha$, which are all classifiable by $\Delta^\omega$ by Proposition \ref{pr:fine_universal}. We will repeatedly use the easy calculations described immediately before Proposition \ref{prop:upper_bounds} repeatedly without explicit reference.

\begin{claim}\label{cl:D_not_attained}
Suppose $D(\BPi^0_n)$ is not attained for some $n \ge 3$. Then $\sim_\alpha$ is potentially $\BPi^0_n$ for every $\alpha < \omega_1$.
\end{claim}

\begin{proof}
If $D(\BPi^0_n)$ is not attained, then $\sim^T_{n-2, k}$ is potentially $\BPi^0_n$ for every $k$. Thus by its definition, we compute that $\sim^T_{n-1}$ is also potentially $\BPi^0_n$. In general, for any $\alpha$, if $\sim^T_{\beta, k}$ is potentially $\BPi^0_n$ for every $\beta < \alpha$ and $k \in \omega$, then so is $\sim^T_\alpha$. By induction, the claim follows.
\end{proof}

\begin{claim}\label{cl:Pi_not_attained}
Suppose $\BPi^0_n$ is not attained for some $n \ge 3$. Then for every $\alpha < \omega_1$, $\sim_\alpha$ is potentially $\BSigma^0_{n}$, and in particular, if $n \ge 4$, then for any $\alpha < \omega_1$, $\sim^T_\alpha$ is potentially $D(\BPi^0_{n-1})$.
\end{claim}

\begin{proof}
If $\BPi^0_n$ is not attained, then $\sim^T_{n-2}$, which is potentially $\BPi^0_n$, must be potentially $D(\BPi^0_{n-1})$. By its definition, $\sim^T_{n-2, k}$ is potentially $\BSigma^0_n$ and thus potentially $D(\BPi^0_{n-1})$ for every $k$. In general, for any $\alpha$, if $\sim^T_{\beta, k}$ is potentially $D(\BPi^0_{n-1})$ for every $\beta < \alpha$ and $k \in \omega$, then $\sim^T_\alpha$ is potentially $\BPi^0_n$ and thus by the assumption must be potentially $D(\BPi^0_{n-1})$.
\end{proof}

\begin{claim}
Suppose $\bigoplus_{\alpha < \lambda}\BPi^0_\alpha$ is not attained for some limit $\lambda$. Then there is some $\alpha < \lambda$ such that for every $\beta < \omega_1$, $\sim_\beta^T$ is potentially $\BPi^0_\alpha$.
\end{claim}

\begin{proof}
If $\bigoplus_{\alpha < \lambda}\BPi^0_\alpha$ is not attained, then there must be some $\alpha < \lambda$ such that $\sim^T_{<\lambda}$ is potentially $\BPi^0_\alpha$. In particular, for every $\beta < \lambda$, we have $\sim^T_\beta$ is potentially $\BPi^0_\alpha$. Thus, so is $\sim_\lambda$. In general, for any $\gamma$, if $\sim^T_\beta$ is potentially $\BPi^0_\alpha$ for every $\beta < \gamma$, then $\sim^T_\gamma$ is potentially $\BPi^0_{\alpha+1}$ by its definition. But $\alpha+ 1 < \lambda$ and so $\sim^T_\gamma$ is potentially $\BPi^0_\alpha$ as well.
\end{proof}

\begin{claim}
Suppose $\BPi^0_\lambda$ is not attained for some limit $\lambda$. Then $\sim^T_\alpha$ is potentially $\bigoplus_{\beta < \lambda} \BPi^0_\beta$ for every $\alpha < \omega_1$.
\end{claim}

\begin{proof}
If $\BPi^0_\lambda$ is not attained, then $\sim^T_\lambda$ is potentially $\bigoplus_{\alpha < \lambda}\BPi^0_\alpha$. We argue by induction that $\sim^T_\alpha$ is potentially $\BPi^0_\lambda$ for every $\alpha < \omega_1$. The claim will then follow.

Fix some $\alpha \ge \lambda$ such that $\sim^T_\alpha$ is potentially $\BPi^0_\lambda$. Then by assumption there is a Borel partition $\mcC^T = \bigsqcup_{\beta < \lambda} \Xspace_\beta$ into invariant sets such that on $\sim^T_\alpha \upharpoonright X_\beta$ is potentially $\BPi^0_\beta$. Then observe that on $\Xspace_\beta$, $\sim^T_{\alpha+1}$ is potentially $\BPi^0_{\beta+1}$, and thus potentially $\BPi^0_\lambda$ because $\beta+2 \le \lambda$. Since $\sim^T_{\alpha+1}$ is potentially $\BPi^0_\lambda$ on each $\Xspace_\beta$, we have that $\sim^T_{\alpha+1}$ as a whole is potentially $\BPi^0_\lambda$.

For the limit case $\lambda'$, simply observe that if $\sim_\beta$ is potentially $\BPi^0_\lambda$ for every $\beta < \lambda'$, then by definition $\sim^T_{\lambda'}$ is too.
\end{proof}

\begin{claim}
Suppose $\BSigma^0_{\lambda+1}$ is not attained for some limit $\lambda$. Then $\sim^T_\alpha$ is potentially $\BPi^0_\lambda$ for every $\alpha < \omega_1$.
\end{claim}

\begin{proof}
This is essentially the same as Claim \ref{cl:D_not_attained}.
\end{proof}

\begin{claim}
Suppose $D(\BPi^0_{\lambda+n})$ is not attained for some limit $\lambda$ and finite $n \ge 2$. Then $\sim^T_\alpha$ is potentially $\BPi^0_{\lambda+n}$ for every $\alpha < \omega_1$.
\end{claim}

\begin{proof}
This is essentially the same as Claim \ref{cl:D_not_attained}.
\end{proof}

\begin{claim}
Suppose $\BPi^0_{\lambda+n}$ is not attained for some limit $\lambda$ and finite $n \ge 2$. Then for every $\alpha < \omega_1$, $\sim^T_\alpha$ is potentially $\BSigma^0_{\lambda+n-1}$, and in particular, for $n \ge 3$, $\sim^T_\alpha$ is potentially $D(\BPi^0_{\lambda+n})$.
\end{claim}

\begin{proof}
This is essentially the same as Claim \ref{cl:Pi_not_attained}.
\end{proof}

The theorem follows from the above claims, using throughout the result from \cite{HKL_1998} that the list $(*)$ of potential complexity classes is exhaustive for Borel orbit equivalence relations induced by Borel actions of non-Archimedean Polish groups.
\end{proof}

We can now conclude Corollary \ref{cor:tsi_complexity_spectrum}. 

\tsicomplexityspectrumcor

\begin{proof}
By \cite{Solecki_1995}, the groups $\mathbb{Z}^\omega$ and $\mathbb{Z}_p^\omega$ are not tame for any prime $p$. By \cite[Theorem 5.11]{Hjorth_1998}, since these groups are abelian, they have Borel orbit equivalence relations of unbounded complexity in $(*)$. Since $(F_2)^\omega$ involves these groups, we know by \cite[Theorem 2.3.5]{Becker_Kechris_1996} that every orbit equivalence relation induced by a continuous action of either of these groups on a Polish space is Borel bi-reducible with such an action of $(F_2)^\omega$, and so $(F_2)^\omega$ has full Borel potential complexity spectrum as well. Then the previous theorem implies that these groups have full Borel complexity spectrum.
\end{proof}

\newpage{}

\bibliography{main}
\bibliographystyle{alpha}
\end{document}